\documentclass[11pt, draft]{amsart}
\usepackage{amsmath, amssymb, epsfig}
\usepackage{amsfonts}
\usepackage{mathrsfs}
\usepackage[arrow,matrix,curve,cmtip,ps]{xy}

\usepackage{amsthm}

\allowdisplaybreaks

		\def\cprime{$'$}

	\newcommand{\ftn}[3]{ #1 : #2 \rightarrow #3 }
	\newcommand{\setof}[2]{\ensuremath{\left\{ #1 \: : \: #2 \right\}}}
	\newcommand{\ilat}{\operatorname{Lat^\mathrm{alg}}}
	\newcommand{\ilatcl}{\operatorname{Lat^\mathrm{top}}}

	\newcommand{\ktop}{K^{\mathrm{top}}}
	
	\newcommand{\kalg}{K^{\mathrm{alg}}}
	\newcommand{\okalg}{K^{\mathrm{alg},+}}

	\newcommand{\iktop}{K^{\mathrm{top}}_{\mathrm{ideal}}}
	\newcommand{\ikalg}{K^{\mathrm{alg}}_{\mathrm{ideal}}}

	\newcommand{\oiktop}{K^{\mathrm{top},+}_{\mathrm{ideal}}}
	\newcommand{\oikalg}{K^{\mathrm{alg},+}_{\mathrm{ideal}}}

	\newcommand{\diag}{\ensuremath{\operatorname{diag}}}

	\newcommand{\Z}{\mathbb{Z}}
	\newcommand{\C}{\mathbb{C}}
	\newcommand{\Q}{\mathbb{Q}}
	
	\newcommand{\N}{\mathbb{N}}
	
	\newcommand{\K}{\mathbb{K}}

	\theoremstyle{plain}
	\newtheorem{thm}{Theorem}[section]
	\newtheorem{lemma}[thm]{Lemma}
	\newtheorem{theor}[thm]{Theorem}
	\newtheorem{propo}[thm]{Proposition}
	\newtheorem{corol}[thm]{Corollary}

	\theoremstyle{definition}
	\newtheorem{defin}[thm]{Definition}
	
	\newtheorem{remar}[thm]{Remark}
	\newtheorem{notat}[thm]{Notation}

	\numberwithin{equation}{section}
	\numberwithin{figure}{section}

\numberwithin{equation}{section}

\begin{document}
\title[Ideal-related $K$-theory]{Ideal-related $K$-theory for Leavitt path algebras and graph $C^*$-algebras}

	\author{Efren Ruiz}
        \address{Department of Mathematics\\University of Hawaii,
Hilo\\200 W. Kawili St.\\
Hilo, HI\\
96720-4091 USA}
        \email{ruize@hawaii.edu}
        
        \author{Mark Tomforde}
        \address{Department of Mathematics\\University of Houston\\
Houston, TX\\
77204- 3008, USA}
        \email{tomforde@math.uh.edu}

\thanks{This work was partially supported by a grant from the Simons Foundation (\#210035 to Mark Tomforde)}

\date{\today}


\begin{abstract}
We introduce a notion of ideal-related $K$-theory for rings, and use it to prove that if two complex Leavitt path algebras $L_\C(E)$ and $L_\C(F)$ are Morita equivalent (respectively, isomorphic), then the ideal-related $K$-theories (respectively, the unital ideal-related $K$-theories) of the corresponding graph $C^*$-algebras $C^*(E)$ and $C^*(F)$ are isomorphic.  This has consequences for the ``Morita equivalence conjecture" and ``isomorphism conjecture" for graph algebras, and allows us to prove that when $E$ and $F$ belong to specific collections of graphs whose $C^*$-algebras are classified by ideal-related $K$-theory, Morita equivalence (respectively, isomorphism) of the Leavitt path algebras $L_\C(E)$ and $L_\C(F)$ implies strong Morita equivalence (respectively, isomorphism) of the graph $C^*$-algebras $C^*(E)$ and $C^*(F)$.  We state a number of corollaries that describe various classes of graphs where these implications hold.  In addition, we conclude with a classification of Leavitt path algebras of amplified graphs similar to the existing classification for graph $C^*$-algebras of amplified graphs.
\end{abstract}

\maketitle

\section{Introduction}

In \cite{gamt:isomorita} Gene Abrams and the second named author examined the relationship between the structures of the Leavitt path algebra and the graph $C^*$-algebra of a given graph.  In \cite[\S9]{gamt:isomorita} the authors made two conjectures: The \emph{Morita equivalence conjecture for graph algebras} states that if $E$ and $F$ are graphs and $L_\C(E)$ is Morita equivalent to $L_\C(F)$, then $C^*(E)$ is strongly Morita equivalent to $C^*(F)$.  The \emph{isomorphism conjecture for graph algebras} states that if $E$ and $F$ are graphs and $L_\C(E) \cong L_\C(F)$ (as rings), then $C^*(E) \cong C^*(F)$ (as $*$-algebras).  

In \cite{gamt:isomorita} the authors showed that the Morita equivalence conjecture and the isomorphism conjecture hold for row-finite graphs whose associated $C^*$-algebras are simple.  This was accomplished by arguing that if $L_\C(E)$ is Morita equivalent to $L_\C(F)$, then the algebraic $K$-theories of $L_\C(E)$ and $L_\C(F)$ are isomorphic, which implies the topological $K$-theories of $C^*(E)$ and $C^*(F)$ are isomorphic, and then classification theorems for simple $C^*$-algebras imply that $C^*(E)$ and $C^*(F)$ are strongly Morita equivalent.  A similar argument can be made using isomorphism in place of Morita equivalence by keeping track of the position of the class of the unit in the $K$-group.  The nontrivial parts of this argument involve (1) showing that the algebraic $K$-theories of $L_\C(E)$ and $L_\C(F)$ are isomorphic implies the topological $K$-theories of $C^*(E)$ and $C^*(F)$ are isomorphic, and (2) applying existing classification theorems for simple $C^*$-algebras.

Recently, classification theory for $C^*$-algebras has made a number of accomplishments in classifying particular collections of nonsimple $C^*$-algebras, and graph $C^*$-algebras have provided a fertile testing grounds in which many of these theorems may be applied.  Indeed, there is now a large list of distinct classes of graph $C^*$-algebras that are classified by $K$-theoretic information.  In all cases, the classifying invariant has been ideal-related $K$-theory, which consists of all cyclic six-term exact sequences of (topological) $K$-groups for all subquotients and the natural transformations between them.  (In the literature, ideal-related $K$-theory is also known by a variety of other names, including ``filtered $K$-theory" and ``filtrated $K$-theory".)

The purpose of this article is to define a notion of ideal-related algebraic $K$-theory for rings that applies to Leavitt path algebras, and to show that if two Leavitt path algebras of graphs satisfying Condition~(K) have isomorphic ideal-related algebraic $K$-theories, then the corresponding graph $C^*$-algebras have isomorphic ideal-related topological $K$-theories.  This result allows one to verify both the Morita equivalence conjecture and the isomorphism conjecture for graphs whose $C^*$-algebras are classified by ideal-related topological $K$-theory, and we describe many of these specific classes in a number of corollaries to our main results.  Throughout, we need the graphs to satisfy Condition~(K) to ensure, among other things, that the ideals in the graph algebra correspond to saturated hereditary subsets in the graph.  (We mention that at the time this is paper is written,  all existing classification results for graph $C^*$-algebras require Condition~(K) as well, so this is a mild hypothesis in the context of the existing theory.)

In addition to verifying the Morita equivalence conjecture and isomorphism conjecture for specific classes of graphs, our results on ideal-related $K$-theory also allow us to give a classification of Leavitt path algebras of amplified graphs with finitely many vertices that is similar to the existing classification for $C^*$-algebras.  Specifically, we show that two such Leavitt path algebras are classified by their ideal-related $K$-theories and also by the transitive closures of their graphs.  This is exactly what happens in the graph $C^*$-algebra case, and hence gives us a converse to the isomorphism conjecture for amplified graphs with finitely many vertices:  If $E$ and $F$ are amplified graphs with a finite number of vertices, then $L_\C(E) \cong L_\C(E)$ (as rings) if and only if $C^*(E) \cong C^*(F)$ (as $*$-algebras).

This paper is organized as follows.  In Section~\ref{Notation-Prelim-sec} we discuss notation and preliminaries regarding graphs, graph algebras, and algebraic $K$-theory of rings.  In Section~\ref{ideal-rel-def-sec} we introduce a definition for the ideal-related algebraic $K$-theory of a ring.  We face some obstacles that do not occur in the $C^*$-algebra setting: first of all, we need to consider all $K$-groups (not just the $K_0$-group and $K_1$-group) due to lack of Bott periodicity, and second we need to require that the ring and all of its subquotients satisfy \emph{excision} in order to ensure that long exact sequences of the $K$-groups exist.  In Section~\ref{mor-equiv-sec} we show in Theorem~\ref{t:induceiso} that if the Leavitt path algebras of two graphs satisfying Condition~(K) have isomorphic ideal-related algebraic $K$-theories, then the $C^*$-algebras of those graphs have isomorphic ideal-related topological $K$-theories.  Using this result, we are able to show in Theorem~\ref{t:meqstrmeq} that the Morita equivalence conjecture for graph algebras holds for any class of graphs whose $C^*$-algebras are classified up to strong Morita equivalence (equivalently, stable isomorphism) by their ideal-related topological $K$-theory.  In Section~\ref{Iso-Con-sec} we consider the isomorphism conjecture for graph algebras, which requires us to examine the position of the class of the unit in $K$-theory.  We show in Theorem~\ref{t:induceisounit} that if the ideal-related algebraic $K$-theories of two unital Leavitt path algebras of graphs satisfying Condition~(K) are isomorphic via an isomorphism taking the class of the unit to the class of the unit, then there is an isomorphism between the ideal-related topological $K$-theories of the graph $C^*$-algebras taking the class of the unit to the class of the unit.  In Section~\ref{amplified-LPA-classification-sec} we use our results on ideal-related $K$-theory to prove a classification result for amplified graphs. 

\smallskip

\noindent \textsc{Acknowledgements:} Part of the research for this article was carried out during the authors visit to the Centre de Recerca Matem\`{a}tica during the thematic program ``The Cuntz Semigroup and the Classification of $C^{*}$-algebras''.  The authors are indebted to this institution for its hospitality.

\section{Notation and Preliminaries} \label{Notation-Prelim-sec}

All of the graphs that we consider will be directed graphs with countably many edges and vertices.  We need this countability hypothesis to ensure that all our graph $C^*$-algebras are separable, which is a necessary hypothesis for many of the classification results that we use.  The countability hypothesis also ensures that the Leavitt path algebras of our graphs have countable sets of enough idempotents, which allows us to use the fact that $L_K(E)$ is Morita equivalent to $L_K(F)$ if and only if $\textsf{M}_\infty (L_K(E)) \cong \textsf{M}_\infty(L_K(F))$ (see \cite[Definition~9.6, Definition~9.9, and Proposition~9.10]{gamt:isomorita}).

\begin{defin}
A \emph{directed graph} $E := (E^0, E^1, r_E, s_E)$ consists of a countable set of vertices $E^0$, a countable set of edges $E^1$, and maps $r_E: E^1 \to E^0$ and $s_E:E^1 \to E^0$ identifying the range and source of each edge.  When we have only one graph, or it is clear from context, we will often drop the subscript on the range and source maps and simply write $r := r_E$ or $s := s_E$.
\end{defin}

\begin{defin}
Let $E = (E^0, E^1, r_E, s_E)$ be a graph.  A \emph{path} of $E$ is a finite sequence of edges $\alpha := e_1 e_2 \ldots e_n$ with $r_E(e_i) = s_E(e_{i+1})$ for $1 \leq i \leq n$, and we call the natural number $n$ the \emph{length} of the graph.  We also consider the vertices to be paths of length zero.  We extend the range and source maps to the collection of paths as follows:  If $\alpha = e_1 e_2 \ldots e_n$ is a path of positive length, we set $s_E(\alpha) = s_E(e_1)$ and $r_E(\alpha) = r(e_n)$, and if $\alpha = v \in E^0$ is a path of length zero we set $s_E(v) = r_E(v) = v$.

A \emph{cycle} in $E$ is a path $\alpha := e_1 \ldots e_n$ of positive length with $s_E(e_1) = r_E(e_n)$.  We call the vertex $s_E(e_1)$ the \emph{base point} of the cycle $\alpha$.  A cycle is called a \emph{simple cycle} if $r_E(e_i) \neq s_E(e_1)$ for $1 \leq i \leq n-1$.  A graph $E$ is said to satisfy \emph{Condition (K)} if no vertex of $E$ is the base point of exactly one simple cycle.

A graph is called an \emph{amplified graph} if for every two vertices $v, w \in E^0$ there are either no edges from $v$ to $w$ or there are countably many edges from $v$ to $w$.  (Equivalently, a graph is an amplified graph if whenever there is an edge between one vertex $v$ to another vertex $w$, then there are infinitely many edges from $v$ to $w$.)
\end{defin}

\begin{defin}
Let $E = ( E^{0} , E^{1} , r_{E} , s_{E} )$ be a graph and let $R$ be a ring.  A collection $\setof{ v, e, e^{*} }{ v \in E^{0}, e\in E^{1} } \subseteq R$ is a \emph{Leavitt $E$-family in $R$} if $\setof{ v }{ v \in E^{0} }$ is a collection of pairwise orthogonal idempotents and the following conditions are 
satisfied: 
\begin{itemize}
\item[(1)] $s_{E}(e)e = er_{E}(e) = e$ for all $e \in E^{1}$; 
\item[(2)] $r_{E}(e)e^{*} = e^{*} s_{E}(e) = e^{*}$ for all $e \in E^{1}$; 
\item[(3)] $e^{*} f = \delta_{e,f} r_{E}(e)$ for all $e, f \in E^{1}$; and 
\item[(4)] $v = \sum_{ e \in s_{E}^{-1} ( \{ v \} ) } ee^{*}$ whenever $0 < | s_{E}^{-1} ( \{ v \} )| < \infty$. 
\end{itemize}
If $K$ is a field, the \emph{Leavitt path algebra of $E$ with coefficients in $K$}, denoted by $L_{K} (E)$, is the universal $K$-algebra generated by a Leavitt $E$-family.
\end{defin}

\begin{defin}
Let $E = ( E^{0} , E^{1} , r_{E} , s_{E} )$ be a graph.  A collection 
\begin{align*}
\setof{ p_{v}, s_{e} }{ v \in E^{0}, e \in E^{1} }
\end{align*}
in a $C^{*}$-algebra $\mathfrak{A}$ is a \emph{Cuntz-Krieger $E$-family in $\mathfrak{A}$} if $\setof{ p_{v} }{ v \in E^{0} }$ consists of pairwise orthogonal projections, $\setof{ s_{e} }{ v \in E^{0} }$ is a collection of partial isometries, and the following conditions are 
satisfied: 
\begin{itemize}
\item[(CK1)] $s_{e}^{*} s_{f}  = \delta_{e,f} p_{ r_{E}(e) }$ for all $e,f \in E^{1}$; 
\item[(CK2)] $s_{e} s_{e}^{*} \leq p_{ s_{E} (e) }$ for all $e \in E^{1}$; and
\item[(CK3)] $p_{v} = \sum_{ e \in s_{E}^{-1} ( \{ v \} ) } s_{e}s_{e}^{*}$ whenever $0 < | s_{E}^{-1} ( \{ v \} )| < \infty$. 
\end{itemize}
The graph $C^{*}$-algebra, denoted by $C^{*} (E)$, is defined to be the universal $C^{*}$-algebra generated by a Cuntz-Krieger $E$-family.
\end{defin}

\begin{defin}
Let $E$ be a graph and let $v, w$ be in $E^{0}$.  We write $v \geq w$ if there exists a path from $v$ to $w$.  A subset $H$ of $E^{0}$ is \emph{hereditary} if for each $v, w \in E^{0}$ with $v \geq w$, $v \in H$ implies that $w \in H$.  A hereditary subset $H$ of $E^{0}$ is \emph{saturated} if for each $v \in E^{0}$ with $0 < | s_{E}^{-1} ( \{ v \} ) | < \infty$, we have
\begin{align*}
r_{E} ( s_{E}^{-1} ( \{ v \} ) ) \subseteq H \text{ implies } v \in H.
\end{align*} 
We will denote the lattice of saturated hereditary subsets of $E^{0}$ by $\mathcal{H}( E )$.  Note that we always have $\emptyset \in \mathcal{H}(E)$ and $E^0 \in \mathcal{H}(E)$.
\end{defin}

\begin{notat}
Let $E$ be a graph and let $H$ be an element of $\mathcal{H} (E )$.  We write $I_{H}^{\mathrm{alg} }$ to denote the two-sided ideal of $L_{\C} ( E )$ generated by $\setof{ v }{ v \in H }$, and we write $I_{H}^{ \mathrm{top} }$ to denote the closed two-sided ideal of $C^{*} ( E )$ generated by $\setof{ p_{v} }{ v \in H }$.
\end{notat}

We define a conjugate linear involution on $L_{ \mathbb{C} } ( E )$ by 
\begin{align*}
\left( \sum \lambda_{i} \alpha_{i} \beta_{i}^{*} \right)^{*} = \sum \overline{ \lambda_{i} } \beta_{i} \alpha_{i}^{*}.
\end{align*} 
With this involution $L_{\mathbb{C} } (E)$ is a complex $*$-algebra with the universal property that if $\mathfrak{A}$ is a complex $*$-algebra and $\setof{ a_{v} , b_{e} }{ v \in E^{0} , e \in E^{1} } \subseteq \mathfrak{A}$ is a set of elements satisfying 

\begin{itemize}
\item[(1)] the $a_{v}$'s are pairwise orthogonal and $a_{v} = a_{v}^{2} = a_{v}^{*}$ for all $v \in E^{0}$;

\item[(2)] $a_{ s_{E}(e) } b_{e} = b_{e} a_{ r_{E}(e) } = b_{e}$ for all $e \in E^{1}$;

\item[(3)] $b_{e}^{*} b_{f} = \delta_{e,f } a_{ r_{E}(e) }$ for all $e, f \in E^{1}$; and

\item[(4)] $a_{v} =  \sum_{ e \in s_{E}^{-1} ( \{v \} ) }  b_{e} b_{e}^{*}$ whenever $0 < | s_{E}^{-1} ( \{ v \} ) | < \infty$,
\end{itemize} 
then there exists a unique algebra $*$-homomorphism $\ftn{ \phi }{ L_{\mathbb{C} } ( E ) }{ \mathfrak{A} }$ satisfying $\phi ( v ) = a_{v}$ and $\phi ( e ) = b_{e}$ for all $v \in E^{0}$ and $e \in E^{1}$.

\begin{theor}\label{t:densesubalg}\cite[Theorem~7.3]{mt:idealst}
$ $
\begin{itemize}
\item[(1)] For any graph $E$, there exists an injective algebra $*$-homomorphism 
\begin{align*}
\ftn{ \iota_{E} }{ L_{ \mathbb{C} } ( E ) }{ C^{*} ( E ) }
\end{align*}
with $\iota_{E} ( v ) = p_{v}$ and $\iota_{E} ( e ) = s_{e}$ for all $v \in E^{0}$ and $e \in E^{1}$.  

\item[(2)] If $E$ is a row-finite graph satisfying Condition (K), then the map $I_{H}^{\mathrm{alg} } \mapsto I_{H}^{ \mathrm{top} }$ is a lattice isomorphism from the lattice of two-sided ideals of $L_\C(E)$ onto the lattice of closed two-sided ideals of $C^*(E)$.  Moreover, the closure of $\iota_{E} ( I_{H}^{\mathrm{alg} }  )$ is equal to $I_{H}^{\mathrm{top}}$ for all $H \in \mathcal{H} ( E )$.
\end{itemize}
\end{theor}

\begin{remar}
In light of Theorem~\ref{t:densesubalg}, we shall use the map $\ftn{ \iota_{E} }{ L_{ \mathbb{C} } ( E ) }{ C^{*} ( E ) }$ to identify the set of generators $\setof{ v, e , e^{*} }{ v \in E^{0} , e \in E^{1} } \subseteq L_{\C}(E)$ with the set of generators $
\setof{ p_{v} , s_{e}, s_{e}^{*} }{ v \in E^{0}, e \in E^{1} } \subseteq C^{*} (E)$, and write $L_{\C} (E) \subseteq C^{*} (E)$ when we do so. 
\end{remar}

\begin{defin}
For every $H \in \mathcal{H} ( E )$, define $\ftn{ \iota_{E,H} }{ I_{H}^{ \mathrm{alg} } }{ I_{H}^{\mathrm{top} } }$ by $\iota_{E,H} = ( \iota_{E} ) \vert_{H}$.  
\end{defin}

Note that if $H_{1}, H_{2} \in \mathcal{H} ( E )$ with $H_{1} \subseteq H_{2}$, then the diagram
\begin{equation*}
\xymatrix{
0 \ar[r] & I_{H_{1}}^{ \mathrm{alg} } \ar[r] \ar[d]^{ \iota_{E, H_{1} } } & I_{H_{2}}^{ \mathrm{alg} } \ar[r] \ar[d]^{ \iota_{E, H_{2} } } &  I_{H_{2}}^{\mathrm{alg} } / I_{H_{1}}^{\mathrm{alg} } \ar[r] \ar[d]^{ \iota_{E, H_{2}/H_{1} } }  & 0 \\
0 \ar[r] & I_{H_{1}}^{ \mathrm{top} } \ar[r] & I_{H_{2}}^{ \mathrm{top} } \ar[r] & I_{H_{2}}^{\mathrm{top} } / I_{H_{1}}^{\mathrm{top} } \ar[r] & 0 
}
\end{equation*}
commutes.

For a ring $R$ we let $\kalg_{*} ( R )$ denote the algebraic $K$-theory of a ring, $\textit{KH}_{*} ( R)$ denote the homotopy algebraic $K$-theory introduced by C.~Weibel in \cite{cb:homtopalgkthy}.  Recall that there is a comparison map $\kalg_{*} ( R ) \rightarrow \textit{KH}_{*} ( R )$ (see \cite{cb:homtopalgkthy} for details).  For a Banach algebra $\mathfrak{A}$ we let $K_{*}^{ \mathrm{top} } ( \mathfrak{A} )$ denote the topological $K$-theory of $\mathfrak{A}$.  

\begin{defin}
Let $H_{1}, H_{2} \in \mathcal{H} ( E )$ with $H_{1} \subseteq H_{2}$.  We define the comparison map $\gamma_{n, H_{2} / H_{1} }^{E} : \kalg_{n} ( I_{H_{2}}^{\mathrm{alg}} / I_{H_{1}}^{\mathrm{alg}}  ) \to  \ktop_{n} (I_{H_{2}}^{\mathrm{top}} /  I_{H_{1}}^{\mathrm{top}} )$ to be the composition 
\begin{align*}
\kalg_{n} ( I_{H_{2}}^{\mathrm{alg}} / I_{H_{1}}^{\mathrm{alg}}  ) \rightarrow \kalg_{n} (I_{H_{2}}^{\mathrm{top}} / I_{H_{1}}^{\mathrm{top}}  ) \rightarrow \textit{KH}_{n} (I_{H_{2}}^{\mathrm{top}} /  I_{H_{1}}^{\mathrm{top}} ) \rightarrow \ktop_{n} (I_{H_{2}}^{\mathrm{top}} /  I_{H_{1}}^{\mathrm{top}} ).
\end{align*}
When $H_{1} = \emptyset$ we write $\gamma_{n, H_{2} }^{E} := \gamma_{n, H_{2} / H_{1} }^{E}$. 
When $H_{2} = E^0$ and $H_{1} = \emptyset$, we write $\gamma_{n}^{E} := \gamma_{n, H_{2} / H_{1} }^{E}$. 
\end{defin}

\begin{remar}\label{r:cmap}
Recall that if $\mathfrak{A}$ is a unital $C^{*}$-algebra, then there exists a surjective homomorphism 
\begin{equation*}
\kalg_{1}( \mathfrak{A} ) = \mathrm{GL} ( \mathfrak{A} ) / [ \mathrm{GL} ( \mathfrak{A} ) , \mathrm{GL} ( \mathfrak{A} ) ] \rightarrow \mathrm{GL} ( \mathfrak{A} ) / \mathrm{GL} ( \mathfrak{A} )_{0} = K_{1}^{ \mathrm{top} } ( \mathfrak{A} ).
\end{equation*}
Let $\mathfrak{A}$ be a $C^{*}$-algebra.  Denote the ring obtained by adjoining a unit to $\mathfrak{A}$ by $\widetilde{\mathfrak{A} }^{\mathrm{alg} }$, and denote the $C^{*}$-algebra obtained by adjoining a unit to $\mathfrak{A}$ by $\widetilde{ \mathfrak{A} }^{\mathrm{top} }$.  Note that there exists a natural unital ring homomorphism from $\widetilde{ \mathfrak{A} }^{\mathrm{alg} }$ to $\widetilde{  \mathfrak{A} }^{\mathrm{top} }$ sending $a + n 1 \in \widetilde{ \mathfrak{A} }^{\mathrm{alg} }$ to $a + n 1 \in \widetilde{  \mathfrak{A} }^{\mathrm{top} }$.  This homomorphism induces a homomorphism from $\kalg_{1} ( \widetilde{ \mathfrak{A} }^{\mathrm{alg}} )$ to $\ktop_{1} (  \widetilde{ \mathfrak{A} }^{\mathrm{top} } )$, and the composition
\begin{equation*}
\kalg_{1} ( \mathfrak{A} ) \rightarrow \textit{KH}_{1} ( \mathfrak{A} ) \rightarrow \ktop_{1} ( \mathfrak{A} )
\end{equation*}
is equal to the composition
\begin{equation*}
\kalg_{1} ( \mathfrak{A} ) \rightarrow \kalg_{1} ( \widetilde{ \mathfrak{A} }^{\mathrm{alg} } ) \rightarrow \kalg_{1} ( \widetilde{ \mathfrak{A}}^{\mathrm{top} } ) \rightarrow \ktop_{1} ( \widetilde{ \mathfrak{A} }^{\mathrm{top} } )  = K_{1}^{ \mathrm{top} } ( \mathfrak{A} ).
\end{equation*}
\end{remar}

$ $

\section{Ideal-related $K$-theory} \label{ideal-rel-def-sec}

In this section we recall the definition of ideal-related topological $K$-theory for $C^*$-algebras, and we introduce a definition of ideal-related algebraic $K$-theory for rings.

\begin{defin}
If $R$ is a ring, we let $\ilat (R)$ denote the lattice of all two-sided ideals of $R$.  If $\mathfrak{A}$ is a $C^{*}$-algebra, we let $\ilatcl( \mathfrak{A} )$ denote the lattice of all closed two-sided ideals of $\mathfrak{A}$.  A \emph{subquotient of $R$} is a ring of the  form $I_{2} / I_{1}$ for some $I_{2} , I_{1} \in \ilat (R )$ with $I_{1} \subseteq I_{2}$, and a \emph{subquotient of $\mathfrak{A}$} is a $C^{*}$-algebra of the form $\mathfrak{I}_{2} / \mathfrak{I}_{1}$ for some $\mathfrak{I}_{2} , \mathfrak{I}_{1} \in \ilatcl( \mathfrak{A} )$ with $\mathfrak{I}_{1} \subseteq \mathfrak{I}_{2}$.
\end{defin}

\begin{defin}\label{topKweb}
Let $\mathfrak A$ be a $C^*$-algebra. Note that for any three ideals  $\mathfrak{I}_{1}, \mathfrak{I}_{2}, \mathfrak{I}_{3} \in  \ilatcl(\mathfrak A)$ with $\mathfrak{I}_1 \subseteq \mathfrak{I}_{2} \subseteq \mathfrak{I}_{3}$, the short exact sequence 
\begin{align*}
0 \to \mathfrak{I}_{2}/ \mathfrak{I}_{1} \to \mathfrak{I}_{3}/ \mathfrak{I}_{1} \to \mathfrak{I}_{3}/ \mathfrak{I}_{2} \to 0
\end{align*}
induces a  six-term exact sequence in $K$-theory:
\[
\xymatrix{
\ktop_{0}(\mathfrak{I}_{2}/ \mathfrak{I}_{1} )  \ar[r]^{\iota_*}& \ktop_{0}(\mathfrak{I}_{3} / \mathfrak{I}_{1})\ar[r]^-{\pi_*}&\ktop_{0}(\mathfrak{I}_{3}/\mathfrak{I}_{2}) \ar[d]^{\partial}\\
\ktop_{1} (\mathfrak{I}_{3}/ \mathfrak{I}_{2}) \ar[u]^-{\partial}&\ktop_{1} (\mathfrak{I}_{3}/ \mathfrak{I}_{1})\ar[l]^-{\pi_*}&\ktop_{1}(\mathfrak{I}_{2}/ \mathfrak{I}_{1})\ar[l]^-{\iota_*}.
}
\]
\begin{itemize}
\item[(1)]  We define $\iktop(\mathfrak A)$ of $\mathfrak A$ to be the collection of all $K$-groups thus occurring, equipped with the natural transformations $\{\iota_*,\pi_*,\partial\}$.  

\item[(2)] The \textsl{ideal-related topological $K$-theory} $\oiktop(\mathfrak A)$ of $\mathfrak{A}$ is $\iktop ( \mathfrak{A} )$ together with the positive cone $\ktop_{0} ( \mathfrak{I}_{2} / \mathfrak{I}_{1} )^+$ for all subquotients $\mathfrak{I}_{2}/ \mathfrak{I}_{1}$.
\end{itemize}
If $\mathfrak{A}$ and $\mathfrak{B}$ are $C^{*}$-algebras, we say \emph{$\oiktop (\mathfrak A)$ is isomorphic to $\oiktop(\mathfrak B)$}, and write
$\oiktop(\mathfrak A)\cong \oiktop(\mathfrak B)$, if the following three conditions hold:
\begin{itemize}
\item[(i)] there exists a lattice isomorphism $\ftn{ \beta }{ \ilatcl( \mathfrak{A} ) }{ \ilatcl( \mathfrak{B} ) }$;
\item[(ii)] for each pair of ideals $\mathfrak{I}_{1},  \mathfrak{I}_{2} \in \ilatcl( \mathfrak{A} )$ with $\mathfrak{I}_{1} \subseteq  \mathfrak{I}_{2}$ there exist group isomorphisms
$$
\alpha^{\mathfrak{I}_{1},\mathfrak{I}_{2} }_n:\ktop_{n}( \mathfrak{I}_{2} / \mathfrak{I}_{1} )\to \ktop_n( \beta( \mathfrak{I}_{2} ) / \beta( \mathfrak{I}_{1} ) ) \quad \text{ for $n=0,1$}
$$
with $\alpha^{\mathfrak{I}_{1},\mathfrak{I}_{2} }_0$ an order isomorphism; and 
\item[(iii)] this collection of isomorphisms preserves all natural transformations (described in (1) above).
\end{itemize}
In this case we consider the pair $$\phi := \left( \left\{ \alpha^{\mathfrak{I}_{1},\mathfrak{I}_{2} }_n : \mathfrak{I}_{1},  \mathfrak{I}_{2} \in \ilatcl( \mathfrak{A} ) \text{ with } \mathfrak{I}_{1} \subseteq  \mathfrak{I}_{2}, n =0,1 \right\}, \beta \right)$$ to be a morphism from $\oiktop (\mathfrak A)$ to $\oiktop(\mathfrak B)$, and we write 
\begin{align*}
\phi : \oiktop (\mathfrak A) \to \oiktop(\mathfrak B)
\end{align*}
and call $\phi$ an \emph{isomorphism}.  (The only morphisms we will need to consider are isomorphisms.)
\end{defin}

\begin{remar}
The ideal-related topological $K$-theory that we define above is known by several other names in the literature.  In addition to ``ideal-related $K$-theory" it is also called ``filtered $K$-theory" \cite{gr:ckalg} or ``filtrated $K$-theory" \cite{ERRlinear}, and it is closely related to the $K$-web introduced by Boyle and Huang \cite{Boyle-Huang}.  To make matters more confusing, Meyer and Nest have also defined a notion of  ``filtrated $K$-theory" that has additional natural transformations besides the ones we have included here \cite{filtrated, rasmusmanuel}.  Although defined differently, it is not known whether the ``filtrated $K$-theory" of Meyer and Nest is equivalent to the definition we have given above.
\end{remar}

In analogy with $C^*$-algebras, we wish to introduce a notion of ideal-related $K$-theory for rings.  In general, algebraic $K$-theory does not satisfy Bott periodicity and so we will need all $K$-groups --- not just $K_0$ and $K_1$ --- in our definition.  In addition, when $R$ is a general ring, and $I$ is an ideal of $R$, the short exact sequence $0 \to I \to R \to R/I \to 0$ does not necessarily induce a long exact sequence of algebraic $K$-groups.  In order to ensure  such long exact sequences for our definition, we must restrict our attention to rings in which every subquotient satisfies excision.  We shall show that all Leavitt path algebras have this property, and thus our definition of algebraic ideal-related $K$-theory will suffice for our purposes.

\begin{defin}
A ring $R$ is said to satisfy \emph{excision in algebraic $K$-theory} if whenever $A$ is a unital ring that contains $R$ as a two-sided ideal, the natural map $\kalg_*(R) \to \kalg_*(A,R)$ is an isomorphism, where $\kalg_*(R)$ denotes algebraic $K$-theory of $A$ and $\kalg_*(A,R)$ denotes the relative algebraic $K$-theory of $A$.  The relative algebraic $K$-theory $\kalg_*(A,R)$ is defined so that the relative algebraic $K$-groups satisfy a long exact sequence.  Thus if $R$ satisfies excision in algebraic $K$-theory and $A$ is an arbitrary (not necessarily unital) ring containing $R$ as a two-sided ideal, then the natural sequence 
\begin{align*}
\xymatrix{
\kalg_{n} ( R ) \ar[r]^{ \iota_{*} } & \kalg_{n} ( A ) \ar[r]^{ \pi_{*} } & \kalg_{n} ( A / R ) \ar[r]^{ \partial_{*}} & \kalg_{n-1} ( R ) 
}
\end{align*}
is exact for each $n \in \Z$.
\end{defin}

\begin{defin}\label{algKweb}
Let $R$ be a ring such that every subquotient $I_{2} / I_{1}$ of $R$ satisfies excision in algebraic $K$-theory.  Then for any three ideals $I_1, I_2, I_3 \in \ilat(R)$ with $I_1 \subseteq I_2 \subseteq I_3$, we have an exact sequence
\[
\xymatrix{
\kalg_n(I_{2}/ I_{1} )  \ar[r]^{\iota_*}& \kalg_n(I_{3} / I_{1}) \ar[r]^-{\pi_*}&\kalg_n (I_{3}/I_{2}) \ar[r]^{\partial} & \kalg_{n-1} (I_{3}/ I_{2})
}
\]
for each $n \in \Z$.
\begin{itemize}
\item[(1)]   We define $\ikalg(R)$ to be the collection of all $K$-groups thus occurring, equipped with the natural transformations $\{\iota_*,\pi_*,\partial\}$.  

\item[(2)] The \textsl{ideal-related algebraic $K$-theory} $\oikalg (R)$ of $R$ consists of $\ikalg(R)$ together with the positive cone of $\kalg_{0}( I_{2} / I_{1} )$ for all two-sided ideals $I_{2} , I_{1} \in \ilat( R )$ with $I_{1} \subseteq I_{2}$.
\end{itemize}

Let $R$ and $S$ be rings that each have the property that every subquotient satisfies excision in algebraic $K$-theory.  We will say that \emph{$\oikalg(R)$ is isomorphic to $\oikalg(S)$}, and write $\oikalg(R)\cong \oikalg(S)$, if the following three conditions hold:
\begin{itemize}
\item[(i)]  there exists a lattice isomorphism $\ftn{ \beta }{ \ilat(R) }{ \ilat( S ) }$;

\item[(ii)] for each pair of ideals $I_{1}, I_{2} \in \ilat(R)$ with $I_1 \subseteq I_2$ there exist group isomorphisms
$$
\alpha^{I_{1},I_{2} }_n:\kalg_{n}( I_{2} / I_{1} )\to \kalg_{n}( \beta( I_{2} ) / \beta( I_{1} ) ) \quad \text{for each $n \in \Z$}
$$
with $\alpha^{I_{1},I_{2} }_0$ an order isomorphism; and
\item[(iii)] this collection of isomorphisms preserves all natural transformations.
\end{itemize}
In this case we consider the pair $$\phi := \left( \left\{ \alpha^{I_{1}, I_{2} }_n : I_{1},  I_{2} \in \ilat( R ) \text{ with } I_{1} \subseteq  I_{2}, n \in \Z  \right\}, \beta \right)$$ to be a morphism from $\oikalg (R)$ to $\oikalg(S)$, and we write 
\begin{align*}
\phi : \oikalg (R) \to \oikalg(S)
\end{align*} 
and call $\phi$ an \emph{isomorphism}.  (The only morphisms we will need to consider are isomorphisms.)
\end{defin}

\begin{remar}
If we define $\oiktop( \mathfrak{A} )\equiv \oiktop( \mathfrak{B} )$ when $\oiktop( \mathfrak{A} )\cong \oiktop( \mathfrak{B} )$ via an isomorphism having $\alpha^{ \mathfrak{I}_{1} , \mathfrak{I}_{2} }_{n} = \alpha^{ \mathfrak{J}_{1} , \mathfrak{J}_{2} }_{n}$ whenever $\mathfrak{I}_{2} / \mathfrak{I}_{1} \cong \mathfrak{J}_{2} / \mathfrak{J}_{1}$, and if we likewise define $\oikalg(R)\equiv \oikalg(S)$ when $\oikalg(R)\cong \oikalg(S)$ via an isomorphism having $\alpha^{I_{1} , I_{2} }_{n} = \alpha^{ J_{1} , J_{2} }_{n}$ whenever $I_{2} / I_{1} \cong J_{2} / J_{1}$, then the equivalence $\equiv$ is closer to the definition of isomorphism of filtrated $K$-theory given in \cite{filtrated}.  The results of Section~\ref{mor-equiv-sec} and Section~\ref{Iso-Con-sec} still hold if our isomorphism of ideal-related $K$-theory is replaced by $\equiv$.
\end{remar}

\begin{defin}
Let $R$ be a ring.  A sequence $\{e_n \}_{n=1}^\infty \subseteq R$ is a \emph{countable approximate unit consisting of idempotents} if the following three properties hold:
\begin{itemize}
\item[(1)] each $e_n$ is an idempotent,
\item[(2)] $e_{n+1} e_n = e_n e_{n+1} =  e_n$ for all $n \in \N$, and
\item[(3)] for any $r \in R$ there exists $n \in \N$ such that $e_nr = re_n = r$.
\end{itemize}
\end{defin}

\begin{remar} \label{sigma-unital-dir-lim-rem}
Note that if $R$ is a ring with a countable approximate unit consisting of idempotents $\{ e_n \}_{n=1}^\infty$, then $R = \bigcup_{i=1}^\infty e_i R e_i$ with $e_1Re_1 \subseteq e_2Re_2 \subseteq \ldots$ and each $e_n R e_n$ a unital subring of $R$ with unit $e_n$.  Thus, in this case, $R$ is the increasing union of a sequence of unital subrings, and hence also a direct limit of a sequence of unital rings.
\end{remar}

\begin{remar}
We wish to show that any ring with a countable approximate unit consisting of idempotents satisfies excision in algebraic $K$-theory.  The most general results on excision are due to Suslin in \cite{as:excisionkthy} where he proves that a ring $R$ satisfies excision if and only if
\begin{equation} \label{H'-unital-eq}
\operatorname{Tor}_n^{\tilde{R}} (R, \Z) = 0 \text{ for $n \geq 0$,}
\end{equation} 
where $\tilde{R}$ denotes the unitization of $R$.  A ring $R$ is said to be \emph{$H'$-unital} if it satisfies \eqref{H'-unital-eq}.  If $R$ is torsion-free as a $\Z$-module, then $R$ is $H'$-unital if and only if $R$ is $H$-unital in the sense of Wodzicki \cite{mw:excisioncyclichom}.  Suslin and Wodzicki proved in \cite[Theorem~B]{sw:excisionalgktheory} that if the ring $R$ is a $\Q$-algebra, then $R$ satisfies excision in algebraic $K$-theory if and only if $R$ is $H$-unital.  Unital rings are both $H$-unital and $H'$-unital \cite[Remark~2.2]{abc:kthyleavitt}.  Using these facts we are able to establish the following lemma.
\end{remar}

\begin{lemma}\label{l:excisionlocalunit}
Any ring with a countable approximate unit consisting of idempotents satisfies excision in algebraic $K$-theory.  
\end{lemma}

\begin{proof}
Every unital ring is $H'$-unital (see \cite[Remark~2.2]{abc:kthyleavitt}).  In addition, by \cite[Remark~2.2]{abc:kthyleavitt}, the class of $H'$-unital rings is closed under direct limits.  Hence all rings with a countable approximate unit consisting of idempotents are $H'$-unital (see Remark~\ref{sigma-unital-dir-lim-rem}).  Suslin has proven that a ring satisfies excision in algebraic $K$-theory if and only if the ring is $H'$-unital \cite{as:excisionkthy}.   Thus all rings with a countable approximate unit consisting of idempotents satisfy excision in algebraic $K$-theory.  
\end{proof}

\begin{lemma} \label{l:sqlocalunits}
If $E$ satisfies Condition~(K), then any subquotient of $L_K(E)$ has a countable approximate unit consisting of idempotents.
\end{lemma}

\begin{proof}
Any Leavitt path algebra has a countable approximate unit consisting of idempotents.  Let $I$ and $J$ be ideals of $L_{K} (E)$ such that $I \subseteq J$.  By \cite[Corollary 6.3]{ermt:ideals-graph-algs}, $J$ is isomorphic to $L_{K} (F)$ for some graph $F$.  Hence, $J$ has a countable approximate unit consisting of idempotents, $\{ e_{n} \}_{ n = 1}^{ \infty }$.  Then $\{ e_{n} + I \}_{ n = 1}^{ \infty }$ is a countable approximate unit consisting of idempotents for $J / I$.
\end{proof}

\begin{remar}
It follows from Lemma~\ref{l:excisionlocalunit} and Lemma~\ref{l:sqlocalunits} that if $E$ is a graph satisfying Condition~(K), then any subquotient of $L_{K} (E)$ satisfies excision in algebraic $K$-theory.  Thus, the definition of ideal-related algebraic $K$-theory applies to $L_K(E)$ whenever $E$ satisfies Condition~(K). 
\end{remar}

\begin{remar} \label{ideals-sat-hered-ident-rem}
If $E$ is a row-finite graph satisfying Condition~(K), then the lattice of ideals of the Leavitt path algebra $L_K(E)$ is isomorphic to the lattice $\mathcal{H}(E)$ of saturated hereditary subsets of $E$.  In this case we may identify any ideal in $L_K(E)$ with its corresponding saturated hereditary subset of $E^0$.  If we do this, then whenever $E_1$ and $E_2$ are row-finite graphs satisfying Condition~(K), an isomorphism between $\ikalg ( L_{K} ( E_1 ) )$ and $\ikalg ( L_{K} ( E_2 ) )$ consists of (i) a lattice isomorphism $\ftn{ \beta }{ \mathcal{H} ( E_{1} )  }{  \mathcal{H} ( E_{2} ) }$ together with (ii)  group isomorphisms
$$
\alpha^{H_{1},H_{2} }_n: \kalg_n(I_{H_{2}}^{\mathrm{alg} } / I_{ H_{1} }^{\mathrm{alg} } ) \to \kalg_n(I_{ \beta( H_{2} ) }^{\mathrm{alg} } / I_{ \beta( H_{1} ) }^{\mathrm{alg} } )\quad \text{for each $n \in \Z$}
$$
for each pair $H_{1}, H_{2} \in \mathcal{H} ( E_{1} )$ with $H_1 \subseteq H_2$, that also (iii)  preserves all natural transformations.  Moreover, this is an isomorphism from $\oikalg ( L_{K} ( E_1 ) )$ onto $\oikalg ( L_{K} ( E_2 ) )$ precisely when all the $\alpha^{H_{1},H_{2} }_0$ are also order isomorphisms.

Similarly, if $E$ is a row-finite graph satisfying Condition~(K), then the lattice of closed ideals of the graph $C^*$-algebra $C^*(E)$ is isomorphic to the lattice $\mathcal{H}(E)$ of saturated hereditary subsets of $E$.  In this case we may identify any ideal in $C^*(E)$ with its corresponding saturated hereditary subset of $E^0$.  If we do this, then whenever $E_1$ and $E_2$ are row-finite graphs satisfying Condition~(K), an isomorphism between $\iktop ( C^* ( E_1 ) )$ and $\iktop ( C^* ( E_2 ) )$ consists of (i) a lattice isomorphism $\ftn{ \beta }{ \mathcal{H} ( E_{1} )  }{  \mathcal{H} ( E_{2} ) }$ together with (ii)  group isomorphisms
$$
\alpha^{H_{1},H_{2} }_n: \ktop_n(I_{H_{2}}^{\mathrm{top} } / I_{ H_{1} }^{\mathrm{top} } ) \to \ktop_n(I_{ \beta( H_{2} ) }^{\mathrm{top} } / I_{ \beta( H_{1} ) }^{\mathrm{top} } )\quad \text{for each $n =0,1$}
$$
for each pair $H_{1}, H_{2} \in \mathcal{H} ( E_{1} )$ with $H_1 \subseteq H_2$, that also (iii)  preserves all natural transformations.  Moreover, this is an isomorphism from $\oiktop ( C^*( E_1 ) )$ onto $\oiktop ( C^*( E_2 ) )$ precisely when all the $\alpha^{H_{1},H_{2} }_0$ are also order isomorphisms.
\end{remar}

\section{The Morita Equivalence Conjecture for graph algebras} \label{mor-equiv-sec}

In this section we show that if two complex Leavitt path algebras of graphs satisfying Condition~(K) have isomorphic ideal-related algebraic $K$-theories, then the corresponding graph $C^*$-algebras have isomorphic ideal-related topological $K$-theories.  This allows us to prove that the Morita equivalence conjecture holds for any class of graphs whose $C^*$-algebras are classified by ideal-related topological $K$-theory.  By applying existing classification theorems, we confirm the Morita equivalence conjecture for a number of specific classes of graphs in the corollaries at the end of this section.

\begin{defin}
Let $M$ be a monoid.  A submonoid $S$ of $M$ is an \emph{order ideal of $M$} if $x+y \in S$ implies $x, y \in S$.  Let $S$ be an order-ideal of $M$.  Define a congruence $\sim_{S}$ on $M$ by setting $a \sim_{S} b$ if and only if there exist $x,y \in S$ such that $a + x = b + y$.  Let $M / S$ be the factor monoid obtained from the congruence.  
\end{defin}

\begin{theor}\cite[Proposition~1.4]{agop:exhangerings}\label{t:exchange}
If $R$ is an exchange ring and $I$ is an ideal of $R$, then $V( R/ I ) \cong V(R) / V(I)$.
\end{theor}

\begin{lemma}\label{l:exchange}
Let $E$ be a row-finite graph satisfying Condition (K).  If $H \in \mathcal{H} ( E )$, and $\overline{E}_{(H, \emptyset) }$ denotes the graph of \cite[Definition~4.1]{ermt:ideals-graph-algs}, then there exists an algebra isomorphism $\phi : I_{H}^{\mathrm{alg} } \to L_{\C} ( \overline{E}_{(H, \emptyset) } )$ and a $*$-isomorphism $\overline{\phi} :  I_{H}^{\mathrm{top} } \to C^{*} ( \overline{E}_{(H, \emptyset) } )$ that make the diagram
\begin{equation*}
\xymatrix{
L_{\C} (\overline{E}_{(H, \emptyset) } ) \ar[r]^-{\phi } \ar[d]_-{ \iota_{ \overline{E}_{(H, \emptyset) } } } & I_{H}^{ \mathrm{alg} } \ar[d]^{ \iota_{E, H } } \\
C^{*} (\overline{E}_{(H, \emptyset) } ) \ar[r]^-{\overline{\phi}} & I_{H}^{ \mathrm{top} }
}
\end{equation*}
commute.  Consequently, we have the following:
\begin{itemize}
\item[(1)] For every $H \in \mathcal{H} ( E )$, the ideal $I_{H}^{\mathrm{alg}}$ is an exchange ring.

\item[(2)] For every $H_{1}, H_{2} \in \mathcal{H} ( E )$ with $H_{1} \subseteq H_{2}$, the natural map from $V( I_{H_{2} }^{\mathrm{alg}})$ to $V( I_{H_{2}}^{\mathrm{alg}} / I_{ H_{1} }^{\mathrm{alg}} )$ is surjective.
\end{itemize}
\end{lemma}

\begin{proof}
The existence of the isomorphisms and commutativity of the diagram follows from the proofs of  \cite[Theorem~5.1]{ermt:ideals-graph-algs} and \cite[Theorem~6.1]{ermt:ideals-graph-algs}.  Statement (1) follows from the fact that $E$ satisfies Condition~(K) implies $\overline{E}_{(H, \emptyset) }$ satisfies Condition~(K), and a graph satisfying Condition~(K) implies the associated Leavitt path algebra is an exchange ring \cite[Theorem~5.8]{AbrPino3}.  Statement (2) then follows from Theorem~\ref{t:exchange}.
\end{proof}

\begin{propo}\label{p:isomv}
Let $E$ be a row-finite graph.   Then for every $H_{1}, H_{2} \in \mathcal{H} ( E )$ with $H_{1} \subseteq H_{2}$, 
\begin{equation*}
\ftn{ V( \iota_{E, H_{2} / H_{1} } ) }{ V( I_{ H_{2} }^{\mathrm{alg} } / I_{H_{1}}^{\mathrm{alg} } ) } { V( I_{ H_{2} }^{\mathrm{top} } / I_{H_{1}}^{\mathrm{top} } ) } 
\end{equation*}
is an isomorphism.  Consequently, 
\begin{align*}
\ftn{ \gamma_{0, H_{2} / H_{1} }^{E} }{ \kalg_{0} ( I_{ H_{2} }^{\mathrm{alg} } / I_{H_{1}}^{\mathrm{alg} }   ) }{ \ktop_{0} (  I_{ H_{2} }^{\mathrm{top} } / I_{H_{1}}^{\mathrm{top} }  ) }
\end{align*}
is an order isomorphism.
\end{propo}

\begin{proof}
We first show the proposition is true for the case that $H_{1} = \emptyset$.  A typical element of $V( I_{H_{2}}^{ \mathrm{top} } )$ has the form $[p]$ for a projection $p \in \mathsf{M}_{n} ( I_{H_{2}}^{ \mathrm{top} } )$.  By the proof of \cite[Theorem~7.1]{amp:nonstablekthy}, $$[p] = [p_{v_{1}}] + [p_{v_{2} }] + \cdots + [p_{v_{n}}]$$ where $v_{i} \in E^{0}$.  Since $p \in \mathsf{M}_{n} ( I_{H_{2}}^{ \mathrm{top} } )$, and $V(I_{H_2}^{ \mathrm{top}})$ is an order ideal, $[p_{v_i}] \in V(I_{H_2}^{ \mathrm{top}})$ for all $1 \leq i \leq n$, and since $I_{H_2}^{ \mathrm{top}}$ is an ideal, $p_{v_i} \in I_{H_2}^{ \mathrm{top}}$ for all $1 \leq i \leq n$.  Thus, $v_{i} \in H_{2}$  for all $1 \leq i \leq n$.  Hence, $[ v_{1} ] + [ v_{2} ] + \dots + [ v_{n} ] \in V( I_{H_{2}}^{ \mathrm{alg} } )$ and 
\begin{equation*}
V( \iota_{E} )( [ v_{1} ] + [ v_{2} ] + \dots + [ v_{n} ] ) = [ p ].
\end{equation*}
Therefore, $V( \iota_{E, H_{2}} )$ is surjective.  

Suppose $x, y \in V( I_{H_{2}}^{\mathrm{alg}} )$ such that $V( \iota_{E} )( x ) = V( \iota_{E} )( y )$.  By \cite[Theorem~7.1]{amp:nonstablekthy}, $x = y$ in $V( L_{\C} ( E ) )$.  Hence, $x = y$ in $V( I_{H_{2}}^{\mathrm{alg}} )$.  Therefore, $V( \iota_{E, H_{2}} )$ is injective.  

We now deal with the general case.  Note that the diagram
\begin{equation*}
\xymatrix{
V( I_{H_{2}}^{\mathrm{alg} } ) \ar[r] \ar[d]_{ V( \iota_{E, H_{2} } ) }  & V( I_{H_{2} }^{\mathrm{alg} } / I_{ H_{1} }^{\mathrm{alg} } ) \ar[d]^{ V( \iota_{E, H_{2} / H_{2} } ) } \\
V( I_{H_{2} }^{ \mathrm{top} } ) \ar[r] & V( I_{H_{2} }^{\mathrm{top} } / I_{ H_{1} }^{\mathrm{top} } ) 
}
\end{equation*}
is commutative.  By Lemma~\ref{l:exchange} and Theorem~\ref{t:exchange}, the horizontal maps are surjective, so that
\begin{align*}
V( I_{H_{2} }^{\mathrm{alg} } / I_{ H_{1} }^{\mathrm{alg} } )  \cong V( I_{H_{2} }^{\mathrm{alg} } ) / V( I_{ H_{1} }^{\mathrm{alg} } ) \quad \text{ and } \quad V( I_{H_{2} }^{\mathrm{top} } / I_{ H_{1} }^{\mathrm{top} } )  \cong V( I_{H_{2} }^{\mathrm{top} } ) / V( I_{ H_{1} }^{\mathrm{top} } ).
\end{align*}
Therefore, $V( \iota_{E, H_{2} / H_{2} } )$ is surjective.  Let $a, b \in V( I_{H_{2}}^{\mathrm{alg} } )$ and let $\overline{a}, \overline{b} \in V( I_{H_{2} }^{\mathrm{alg} } / I_{ H_{1} }^{\mathrm{alg} } )$ be the image of $a$ and $b$ in $V( I_{H_{2} }^{\mathrm{alg} } / I_{ H_{1} }^{\mathrm{alg} } )$, respectively.  Suppose $V( \iota_{E, H_{2} / H_{2} } ) ( \overline{a} ) = V( \iota_{E, H_{2} / H_{2} } ) ( \overline{b} )$.  Then there exist $x, y \in V( I_{ H_{1} }^{\mathrm{top} } )$ such that
\begin{align*}
 V( \iota_{E, H_{2} } ) (a) + x =  V( \iota_{E, H_{2} } ) (b)+ y
\end{align*}
in $V( I_{H_{2} }^{ \mathrm{top} } )$.  Since $V( \iota_{E, H_{1} } )$ is an isomorphism, there exist $c, d \in  V( I_{ H_{1} }^{\mathrm{alg} } )$ such that $V( \iota_{ E , H_{1} }  ) ( c ) = x$ and $V( \iota_{ E , H_{1} }  ) ( d ) = y$.  Since $V( \iota_{ E , H_{2} }  )$ is an isomorphism, we have that $a + c = b + d$.  Hence, $\overline{a} = \overline{b}$.  Therefore, $V( \iota_{E, H_{2} / H_{2} } )$ is injective.

The last statement follows from the fact that $\kalg_{0} ( I_{ H_{2} }^{\mathrm{alg} } / I_{H_{1}}^{\mathrm{alg} } )$ is the enveloping group of $V( I_{ H_{2} }^{\mathrm{alg} } / I_{H_{1}}^{\mathrm{alg} } )$ and $K_{0}( I_{ H_{2} }^{\mathrm{top} } / I_{H_{1}}^{\mathrm{top} } ) = \ktop_{0}(I_{ H_{2} }^{\mathrm{top} } / I_{H_{1}}^{\mathrm{top} } )$ is the enveloping group of $V(I_{ H_{2} }^{\mathrm{top} } / I_{H_{1}}^{\mathrm{top} })$.
\end{proof}

The following theorem is a result of Ara, Brustenga, and Cortinas in \cite{abc:kthyleavitt}.

\begin{theor}\cite[Theorem~7.6 and Theorem~9.1]{abc:kthyleavitt} \label{t:abskthy}
Let $E$ be a row-finite graph and let $\mathsf{A}_{E}$ be the vertex matrix of $E$.  Let $\mathsf{B}_{E}$ be (rectangular) submatrix obtained from $\mathsf{A}_{E}$ by removing the rows corresponding to the sinks of $E$. Then 
\begin{align*}
\kalg ( L_{\C} (E) ) \cong \mathrm{hocofiber} ( \kalg( \C )^{ E^{0 } \setminus \mathrm{Sink} ( E ) } \overset{ 1 - \mathsf{B}_{E}^{t} }{ \longrightarrow }  \kalg( \C )^{E^0} ) \\
\ktop( C^{*}(E) ) \cong \mathrm{hocofiber} ( \ktop( \C )^{ E^{0 } \setminus \mathrm{Sink} ( E ) } \overset{ 1 - \mathsf{B}_{E}^{t} }{ \longrightarrow } \ktop( \C )^{E^0}  ) 
\end{align*}
Consequently, the following chain complexes 
\begin{align*}
\scalebox{.77}{
\xymatrix{
\kalg_{n} ( \C )^{ E^{0} \setminus \mathrm{Sink}(E) } \ar[r]^-{1- \mathsf{B}_{E}^{t} } & \kalg_{n} ( \C )^{ E^{0} } \ar[r] & \kalg_{n} ( L_{\C} ( E ) ) \ar[r] & \kalg_{n-1} ( \C )^{ E^{0} \setminus \mathrm{Sink} ( E ) }  \ar[r]^-{1- \mathsf{B}_{E}^{t} }  & \kalg_{n-1} ( \C )^{ E^{0} } 
}
}
\end{align*}
and
\begin{align*}
\scalebox{.76}{
\xymatrix{
\ktop_{n} ( \C )^{ E^{0} \setminus \mathrm{Sink}(E) } \ar[r]^-{1- \mathsf{B}_{E}^{t} } & \ktop_{n}( \C )^{ E^{0} } \ar[r] & \ktop_{n}( C^{*} ( E ) ) \ar[r] & \ktop_{n-1}( \C )^{ E^{0} \setminus \mathrm{Sink} ( E ) }  \ar[r]^-{1- \mathsf{B}_{E}^{t} }  & \ktop_{n-1} ( \C )^{ E^{0} }
}
}
\end{align*}
are exact for all $n \in \Z$.
\end{theor}

\begin{lemma}\label{l:exactalgtop}
Let $E$ be a graph and let $H_{1}$, $H_{2}$, and $H_{3}$ be elements of $\mathcal{H} ( E )$ such that $H_{1} \subseteq H_{2} \subseteq H_{3}$.  Then the diagrams
\begin{equation*}
\scalebox{.88}{
\xymatrix{
\kalg_{1} ( I_{H_{2}}^{ \mathrm{alg} }/I_{H_{1}}^{ \mathrm{alg} } ) \ar[r] \ar[d]_{ \gamma_{1, H_{2} / H_{1} }^{E} } & \kalg_{1} ( I_{H_{3}}^{ \mathrm{alg} } / I_{H_{1}}^{ \mathrm{alg} } ) \ar[r] \ar[d]_{ \gamma_{1, H_{3} / H_{1} }^{E} } & \kalg_{1} ( I_{H_{3}}^{ \mathrm{alg} } / I_{H_{2}}^{\mathrm{alg} } ) \ar[r] \ar[d]_{ \gamma_{1, H_{3} / H_{2} }^{E} } & \kalg_{0} ( I_{H_{2}}^{ \mathrm{alg} } /  I_{H_{1}}^{ \mathrm{alg}  })  \ar[d]_{ \gamma_{0, H_{2} / H_{1} }^{E} }   \\
\ktop_{1} ( I_{H_{2}}^{ \mathrm{top} }/I_{H_{1}}^{ \mathrm{top} } ) \ar[r] & \ktop_{1} ( I_{H_{3}}^{ \mathrm{top} }/I_{H_{1}}^{ \mathrm{top} }) \ar[r] & \ktop_{1} ( I_{H_{3}}^{ \mathrm{top} } / I_{H_{2}}^{\mathrm{top} } ) \ar[r] & \ktop_{0} ( I_{H_{2}}^{ \mathrm{top} }/I_{H_{1}}^{ \mathrm{top} }) }
}
\end{equation*}
and
\begin{equation*}
\xymatrix{
\kalg_{0} ( I_{H_{2}}^{ \mathrm{alg} } /  I_{H_{1}}^{ \mathrm{alg}  }) \ar[r] \ar[d]_{ \gamma_{0, H_{2} / H_{1} }^{E} }  & \kalg_{0} ( I_{H_{3}}^{ \mathrm{alg} } / I_{H_{1}}^{ \mathrm{alg} }) \ar[r] \ar[d]_{ \gamma_{0, H_{3} / H_{1} }^{E} }  & \kalg_{0} ( I_{H_{3}}^{ \mathrm{alg} } / I_{H_{2}}^{\mathrm{alg} } ) \ar[d]_{ \gamma_{0, H_{3} / H_{2} }^{E} }  \\
\ktop_{0} ( I_{H_{2}}^{ \mathrm{top} }/I_{H_{1}}^{ \mathrm{top} }) \ar[r] & \ktop_{0}( I_{H_{3}}^{ \mathrm{top} }/I_{H_{1}}^{ \mathrm{top} }) \ar[r] & \ktop_{0} ( I_{H_{3}}^{ \mathrm{top} } / I_{H_{2}}^{\mathrm{top} } )
}
\end{equation*}
are commutative.
\end{lemma}

\begin{proof}
This follows from \cite[Theorem~2.4.1]{gc:algkthy} and \cite[Theorem~3.1.9]{gc:algkthy}, and Remark~\ref{r:cmap}.
\end{proof}

\begin{lemma}\label{l:K1}
Let $E$ be a row-finite graph.  For all $H_{1}, H_{2}$ in $\mathcal{H} ( E )$ with $H_{1} \subseteq H_{2}$, we have that $\gamma_{1, H_{2} / H_{1} }^{E}$ is surjective and $\ker( \gamma_{1, H_{2} / H_{1} }^{E} )$ is a divisible group.  
\end{lemma}

\begin{proof}
We first show the lemma is true for the case $H_{2} = E^0$ and $H_{1} = \emptyset$.  Note that from Theorem~\ref{t:abskthy}, the following diagram is commutative
\begin{equation*}
\scalebox{.9}
{\xymatrix{
\kalg_{1} ( \C )^{ E^{0} \setminus \mathrm{Sink}(E) } \ar[r]^-{1- \mathsf{B}_{E}^{t} } \ar[d] & \kalg_{1} ( \C )^{ E^{0} } \ar[r] \ar[d] & \kalg_{1} ( L_{\C} ( E) ) \ar[r] \ar[d]^{ \gamma_{ 1}^{E} } & \Z^{ E^{0} \setminus \mathrm{Sink} ( E) }  \ar[r]^-{1- \mathsf{B}_{E}^{t} }  \ar@{=}[d] & \Z^{ E^{0} } \ar@{=}[d]  \\
0 \ar[r]_-{1- \mathsf{B}_{E}^{t} } & 0 \ar[r] & \ktop_{1}( C^{*} ( E ) ) \ar[r] & \Z^{ E^{0} \setminus \mathrm{Sink} ( E) }  \ar[r]_-{1- \mathsf{B}_{E}^{t} }  & \Z^{ E^{0} }
}
}
\end{equation*}
where the rows are exact.  A diagram chase shows that $\gamma_{1}^{ E }$ is surjective and $\mathrm{coker} ( 1 - \mathsf{B}_{E}^{t} ) \cong \ker  \gamma_{1 }^{E}$.  Since $\kalg_{1} ( \C )^{ E^{0} } \cong ( \C^{\times} )^{ E^{0} }$ is a divisible group, and quotients of divisible groups are divisible, it follows that $\mathrm{coker} ( 1 - \mathsf{B}_{E}^{t} ) \cong \ker  \gamma_{1 }^{E }$ is a divisible group.

We now consider the case $H_{2} = H$ and $H_{1} = \emptyset$.  By Lemma~\ref{l:exchange}, the diagram
\begin{align*}
\xymatrix{
L_{\C} ( \overline{E}_{(H, \emptyset) } ) \ar[r]^-{\cong} \ar[d]_{ \iota_{ \overline{E}_{(H, \emptyset) } } } & I_{H}^{ \mathrm{alg} } \ar[d]^{ \iota_{E, H } } \\
C^{*} ( \overline{E}_{(H, \emptyset) } ) \ar[r]^-{\cong} & I_{H}^{ \mathrm{top} }
}
\end{align*}
is commutative.  Therefore, 
\begin{equation*}
\xymatrix{
\kalg_{1} ( L_{\C} ( \overline{E}_{(H, \emptyset) }  ) ) \ar[r]^-{\cong} \ar[d]_{ \gamma_{1 }^{\overline{E}_{(H, \emptyset) }}  } & \kalg_{1} ( I_{H}^{ \mathrm{alg} } ) \ar[d]^{ \gamma_{1, H }^{ E } } \\
\ktop_{1} ( C^{*} ( \overline{E}_{(H, \emptyset) } ) ) \ar[r]^-{\cong} & \ktop_{1} ( I_{H}^{ \mathrm{top} } )
}
\end{equation*}
is commutative.  As in the previous case, $\gamma_{1}^{\overline{E}_{(H, \emptyset) } }$ is a surjective homomorphism and $\ker(  \gamma_{1 }^{\overline{E}_{(H, \emptyset) }} )$ is a divisible group.  Hence, $\gamma_{ 1, H }^{ E }$ is surjective and $\ker( \gamma_{ 1, H}^{E} )$ is a divisible group.

Finally, suppose $H_{1}$ and $H_{2}$ are elements of $\mathcal{H}( E )$ with $H_{1} \subseteq H_{2}$.  Then by Lemma~\ref{l:exactalgtop}, the diagram
\begin{equation*}
\scalebox{.8}{
\xymatrix{
\kalg_{1} ( I_{H_{1}}^{ \mathrm{alg} } ) \ar[r] \ar[d]_{ \gamma_{1, H_{1} }^{E} } & \kalg_{1} ( I_{H_{2}}^{ \mathrm{alg} } ) \ar[r] \ar[d]_{ \gamma_{1, H_{2} }^{E} } & \kalg_{1} ( I_{H_{2}}^{ \mathrm{alg} } / I_{H_{1}}^{\mathrm{alg} } ) \ar[r] \ar[d]_{ \gamma_{1, H_{2} / H_{1} }^{E} } & K_{0} ( I_{H_{1}}^{ \mathrm{alg} } ) \ar[r] \ar[d]_{ \gamma_{0, H_{1} }^{E} }  & \kalg_{0} ( I_{H_{2}}^{ \mathrm{alg} } )  \ar[d]_{ \gamma_{0, H_{2} }^{E} }   \\
\ktop_{1} ( I_{H_{1}}^{ \mathrm{top} } ) \ar[r] & \ktop_{1} ( I_{H_{2}}^{ \mathrm{top} } ) \ar[r] & \ktop_{1} ( I_{H_{2}}^{ \mathrm{top} } / I_{H_{1}}^{\mathrm{top} } ) \ar[r] & \ktop_{0} ( I_{H_{1}}^{ \mathrm{top} } ) \ar[r] & \ktop_{0}( I_{H_{2}}^{ \mathrm{top} } ) 
}
}
\end{equation*}
is commutative and the rows are exact.  By Proposition~\ref{p:isomv}, $\gamma_{0, H_{1} }^{E}$ and $\gamma_{ 0, H_{2} }^{E}$ are isomorphisms.  From the previous case, $\gamma_{1, H_{2}}^{E}$ is surjective and $\ker( \gamma_{1, H_{2}}^{E} )$ is a divisible group.  A diagram chase shows that $\gamma_{1, H_{2} / H_{1} }^{E}$ is surjective and the map $\kalg_{1} ( I_{H_{2}}^{\mathrm{alg} } ) \rightarrow \kalg_{1} ( I_{H_{2}}^{ \mathrm{alg} } / I_{H_{1}}^{\mathrm{alg} }  )$ maps $\ker ( \gamma_{1, H_{2} }^{E} )$ onto $\ker( \gamma_{1, H_{2} / H_{1} }^{E} )$.  Since $\ker ( \gamma_{1, H_{2} }^{E} )$ is divisible and the quotient of a divisible group is divisible, $\ker( \gamma_{1, H_{2} / H_{1} }^{E} )$ is divisible.
\end{proof}

\begin{lemma}\label{l:divisibleker}
Let $G_{1}$ and $G_{2}$ be abelian groups.  Suppose $H_{i}$ is a subgroup of $G_{i}$ such that $H_{i}$ is a divisible group and $G_{i} / H_{i}$ is a free group for each $i = 1, 2$.  If $\ftn{ \alpha }{ G_{1} }{ G_{2} }$ is an isomorphism, then $\alpha ( H_{1} ) = H_{2}$.  Consequently, the restriction $\ftn{ \alpha \vert_{H_{1} } }{ H_{1} }{ H_{2} }$ is an isomorphism and there exists an isomorphism $\ftn{ \overline{ \alpha } }{ G_{1} / H_{1} }{ G_{2} / H_{2} }$ such that the diagram 
\begin{equation*}
\xymatrix{
0 \ar[r] & H_{1} \ar[r]  \ar[d]^{ \alpha \vert_{H_{1} } } & G_{1} \ar[r] \ar[d]^{ \alpha } & G_{1} / H_{1} \ar[r]  \ar[d]^{ \overline{\alpha} } & 0 \\
0 \ar[r] & H_{2} \ar[r] & G_{2} \ar[r] & G_{2} / H_{2} \ar[r] & 0 
}
\end{equation*} 
is commutative.
\end{lemma}

\begin{proof}
Let $\overline{ \alpha(H_{1}) }$ be the image of $\alpha ( H_{1} )$ in $G_{2} / H_{2}$.  Since $H_{1}$ is a divisible group, $\overline{ \alpha ( H_{1 } ) }$ is a divisible group.  Since $G_{2} / H_{2}$ is a free group, we have that $\overline{ \alpha_{1} ( H_{1 } ) } = 0$.  Thus, $\alpha ( H_{1} ) \subseteq H_{2}$.  Similar, $\alpha^{-1} ( H_{2} ) \subseteq H_{2}$.  The lemma now follows.
\end{proof}

\begin{theor}\label{t:induceiso}
Let $E_{1}$ and $E_{2}$ be row-finite graphs satisfying Condition (K).  If $\oikalg( L_{\C} ( E_{1} ) )$ is isomorphic to $\oikalg ( L_{\C} ( E_{2} ) )$, then $\oiktop( C^{*} (E_{1} ) )$ is isomorphic to $\oiktop ( C^{*} ( E_{2} ) )$.   Moreover, the isomorphism can be chosen such that the diagram
\begin{align*}
\xymatrix{
\kalg_{m} ( L_{\C} ( E_{1} ) ) \ar[r]^{ \cong } \ar[d]_{ \gamma^{E_1}_{m } } & \kalg_{m} ( L_{\C} ( E_{2} ) ) \ar[d]^{ \gamma_{m}^{E_2} } \\
\ktop_{m} ( C^{*} ( E_{1} ) ) \ar[r]_{ \cong} & \ktop_{m} ( C^{*} ( E_{2} ) )
}
\end{align*}
is commutative for $m = 0 , 1$.
\end{theor}

\begin{proof}
Suppose $\oikalg( L_{\C} ( E_{1} ) ) \cong \oikalg ( L_{\C} ( E_{2} ) )$.  By Remark~\ref{ideals-sat-hered-ident-rem}, there exists a lattice isomorphism $\ftn{ \beta }{ \mathcal{H} ( E_{1} ) } { \mathcal{H} ( E_{2} ) }$ and for each pair $H_{1}$ and $H_{2}$ in $\mathcal{H} ( E_{1} )$ with $H_{1} \subseteq H_{2}$, there exist group isomorphisms
\begin{align*}
\ftn{ \lambda_{n}^{H_{1} , H_{2} } }{ \kalg_{n} ( I_{H_{2}}^{\mathrm{alg} }  / I_{ H_{1} }^{\mathrm{alg} } ) }{ \kalg_{n} ( I_{\beta ( H_{2} ) }^{\mathrm{alg} }  / I_{ \beta( H_{1} ) }^{\mathrm{alg} } ) } \text{ for all } n \in \Z
\end{align*}
with $\lambda_{0}^{ H_{1} , H_{2}}$ an order isomorphism, and this collection preserves all natural transformations. 

Let $H_{1} , H_{2} \in \mathcal{H} ( E_{1} )$ with $H_{1} \subseteq H_{2}$.  Since $\lambda_{m}^{ H_{1}, H_{2} }$ is an isomorphism for $m= 0 ,1$, by Proposition~\ref{p:isomv}, Lemma~\ref{l:divisibleker}, and Lemma~\ref{l:K1}, there exists an isomorphism 
\begin{align*}
\ftn{ \alpha_{m }^{ H_{1} , H_{2} } }{ \ktop_{m} ( I_{H_{1}}^{ \mathrm{top} }/I_{H_{2}}^{ \mathrm{top} } ) }{ \ktop_{m} ( I_{ \beta( H_{2} ) }^{ \mathrm{top} }/I_{ \beta( H_{1} ) }^{ \mathrm{top} } ) }
\end{align*}
such that the diagram
\begin{align*}
\xymatrix{
\kalg_{m} \left( I_{H_{2}}^{ \mathrm{alg} }/I_{H_{1}}^{ \mathrm{alg} } \right) \ar[rr]^-{ \lambda_{m}^{H_{1} , H_{2}} }  \ar[d]_{ \gamma_{m, H_{2} / H_{1} }^{E_{1}} } & & \kalg_{m} \left( I_{ \beta( H_{2} ) }^{ \mathrm{alg} }/ I_{ \beta( H_{1} ) }^{ \mathrm{alg} } \right) \ar[d]^{ \gamma_{m, \beta ( H_{2} ) /  \beta( H_{1} ) }^{E_{2}} }  \\
\ktop_{m} \left( I_{ H_{2}}^{ \mathrm{top} }/I_{H_{1}}^{ \mathrm{top} } \right)  \ar[rr]_-{ \alpha_{m}^{ H_{1} , H_{2} } } & & \ktop_{m}\left( I_{ \beta( H_{2} ) }^{ \mathrm{top} }/ I_{ \beta( H_{1} ) }^{ \mathrm{top} } \right) 
}
\end{align*} 
is commutative for $m= 0 ,1$ and $\alpha_{0}^{  H_{1} , H_{1} }$ is an order isomorphism.

Since the collection 
$$\phi := \left( \left\{ \lambda^{H_{1}, H_{2} }_n : H_{1},  H_{2} \in \mathcal{H} ( E_{1} ) \text{ with } H_{1} \subseteq  H_{2}, n \in \Z  \right\}, \beta \right)$$
is an isomorphism from $\oikalg( L_{\C} ( E_{1} ) )$ to $\oikalg ( L_{\C} ( E_{2} ) )$, 
by Lemma~\ref{l:exactalgtop} and Remark~\ref{ideals-sat-hered-ident-rem}, the collection
$$\psi := \left( \left\{ \alpha^{H_{1}, H_{2} }_n : H_{1},  H_{2} \in \mathcal{H} ( E_{1} ) \text{ with } H_{1} \subseteq  H_{2}, n \in \Z  \right\}, \beta \right)$$
is an isomorphism from $\oiktop ( C^{*} ( E_{1} ) )$ to $\oiktop ( C^{*} ( E_{2} ) )$.
\end{proof}

\begin{remar}
It is possible that a result similar to that of Theorem~\ref{t:induceiso} holds whenever we have two $*$-algebras with isomorphic ideal-related algebraic $K$-theories and we pass to their enveloping $C^*$-algebras.  However, the authors do not no how such a general result would be proven.
\end{remar}

\begin{theor}\label{t:meqstrmeq}
Let $\mathcal{C}$ be a class of graphs that satisfies the following two properties:
\begin{itemize}
\item[(1)] Every graph in $\mathcal{C}$ satisfies Condition~(K).
\item[(2)] If $E, F \in \mathcal{C}$ and $\oiktop( C^*(E )) \cong \oiktop (C^*(F))$, then $C^{*} ( E )$ is strongly Morita equivalent to $C^{*} ( F )$.
\end{itemize}
Then the Morita equivalence conjecture holds for all graphs in $\mathcal{C}$.  In other words, if $E, F \in \mathcal{C}$ and $L_{\C} ( E )$ is Morita equivalent to $L_{\C} ( F )$, then $C^{*} ( E )$ is strongly Morita equivalent to $C^{*} ( F )$.
\end{theor}

\begin{proof}
Suppose $E, F \in \mathcal{C}$ and $L_{\C} ( E )$ is Morita equivalent to $L_{\C} ( F )$.  By \cite[Theorem~2.11]{DT1} and \cite[Theorem~5.2]{AbrPino3} there exists a row-finite graph $\widetilde{E}$ satisfying Condition~(K) such that $L_{\C} ( \widetilde{E} )$ is Morita equivalent to $L_\C(E)$ and $C^*(\widetilde{E})$ is strongly Morita equivalent to $C^*(E)$.  (Such a graph $\widetilde{E}$ is called a \emph{desingularization} of $E$.)  Likewise, there exists a row-finite graph $\widetilde{F}$ satisfying Condition~(K) such that $L_{\C} ( \widetilde{F} )$ is Morita equivalent to $L_\C(F)$ and $C^*(\widetilde{F})$ is strongly Morita equivalent to $C^*(F)$.  

Let $S\widetilde{E}$ be the stabilization graph of $\widetilde{E}$, formed by adding an infinite head to each vertex of $\widetilde{E}$.  It follows from \cite[Proposition~9.8]{gamt:isomorita} that $L_\C (S\widetilde{E}) \cong\mathsf{M}_\infty (L_\C(\widetilde{E}))$ (as $*$-algebras) and $C^*( S\widetilde{E}) \cong C^*(\widetilde{E}) \otimes \K$ (as $*$-algebras).  Likewise, if $S\widetilde{F}$ is the stabilization graph of $\widetilde{F}$, then $L_\C (S\widetilde{F}) \cong\mathsf{M}_\infty (L_\C(\widetilde{F}))$ (as $*$-algebras) and $C^*( S\widetilde{F}) \cong C^*(\widetilde{F}) \otimes \K$ (as $*$-algebras).

Since  $L_{\C} ( E )$ is Morita equivalent to $L_{\C} ( F )$, it follows that $L_{\C} ( \widetilde{E} )$ is Morita equivalent to $L_{\C} ( \widetilde{F} )$.  By \cite[Proposition~9.10]{gamt:isomorita} $\mathsf{M}_\infty ( L_{\C} ( \widetilde{E} ) ) \cong \mathsf{M}_\infty ( L_{\C} ( \widetilde{F} ) )$ (as rings).  Thus $L_\C (S\widetilde{E}) \cong L_\C (S\widetilde{F})$ (as rings).  Hence $\oikalg (L_\C(S\widetilde{E})) \cong \oikalg (L_\C(S \widetilde{F}))$.  Since $\widetilde{E}$ and $\widetilde{F}$ are row-finite and satisfy Condition~(K), and since the process of stabilizing a graph preserves row-finiteness and Condition~(K), Theorem~\ref{t:induceiso} implies $\oiktop (C^*(S\widetilde{E})) \cong \oiktop (C^*(S \widetilde{F}))$.

In addition, since $C^*(E)$ is strongly Morita equivalent to $C^*(S\widetilde{E})$, and  $C^*(F)$ is strongly Morita equivalent to $C^*(S\widetilde{F})$, it follows that $$\oiktop (C^*(E)) \cong \oiktop (C^*(S \widetilde{E}))  \ \text{ and } \ \oiktop (C^*(F)) \cong \oiktop (C^*(S \widetilde{F})).$$  Hence $$\oiktop (C^*(E)) \cong \oiktop (C^*(F)).$$
By hypothesis, we then have $C^*(E)$ is strongly Morita equivalent to $C^*(F)$.
\end{proof}

\begin{remar}
Define
$$\mathcal{C} := \{ E : \text{$E$ is a graph and $C^*(E)$ has finitely many ideals} \}$$
and note that $\mathcal{C}$ coincides with the class of graphs $E$ such that $L_\C(E)$ has finitely many ideals.  It follows from basic graph algebra results that every graph in $\mathcal{C}$ satisfies Condition~(K), and hence $\mathcal{C}$ satisfies Property~(1) of Theorem~\ref{t:meqstrmeq}.  It was boldly conjectured in \cite{ERRlinear} that the $C^*$-algebras of the graphs in $\mathcal{C}$ are determined up to stable isomorphism by their ideal-related topological $K$-theory; that is, it was conjectured that $\mathcal{C}$ satisfies Property~(2) of Theorem~\ref{t:meqstrmeq}.  If this conjecture is true, then Theorem~\ref{t:meqstrmeq} implies that the Morita equivalence conjecture of  \cite[\S9]{gamt:isomorita} holds for all graphs in $\mathcal{C}$. 

Although it is not known if all graphs with $C^*$-algebras having finitely many ideals are determined up to stable isomorphism by their  ideal-related topological $K$-theory, special cases of this conjecture have been established for many subclasses of $\mathcal{C}$.  This allows us to establish a number of corollaries using these results.
\end{remar}

\begin{corol} \label{many-classifications-cor}
Let $E$ and $F$ be graphs, and suppose that $L_{\C} ( E )$ is Morita equivalent to $L_{\C} ( F )$.  If any one of the following bulleted points is true:
\begin{itemize}
\item $C^*(E)$ and $C^*(F)$ each have exactly one proper, nonzero ideal.
\item $C^*(E)$ and $C^*(F)$ each have a largest proper ideal that is also AF.
\item $C^*(E)$ and $C^*(F)$ have a smallest nonzero ideal that is purely infinite and whose corresponding quotient is AF.
\item $E$ and $F$ are each finite graphs with no sinks satisfying Condition~(K)
\item $E$ and $F$ are each amplified graphs
\item $C^*(E)$ and $C^*(F)$ are each purely infinite and have primitive ideal spaces that are accordion spaces in the sense of \cite[Definition~1.1]{rasmusmanuel}.
\item $C^*(E)$ and $C^*(F)$ are each simple.
\end{itemize}
then $C^{*} ( E )$ is strongly Morita equivalent to $C^{*} ( F )$.
\end{corol}

\begin{proof}
We first prove the case when the first bulleted point is satisfied.  In this case, let
$$\mathcal{C} := \{ E : \text{$E$ is a graph and $C^*(E)$ has exactly one proper, nonzero ideal} \}.$$  Since every graph whose associated $C^*$-algebra has a finite number of ideals must satisfy Condition~(K), $\mathcal{C}$ satisfies Property~(1) of Theorem~\ref{t:meqstrmeq}.  It follows from \cite[Theorem~4.5]{semt_classgraphalg} that $\mathcal{C}$ satisfies Property~(2) of Theorem~\ref{t:meqstrmeq}.  Thus $C^*(E)$  
is strongly Morita equivalent to $C^{*} ( F )$.

For each of the following bullet points we use a similar strategy to apply Theorem~\ref{t:meqstrmeq}; namely, we define a collection $\mathcal{C}$ of graphs based on properties stated in the bullet point, and then show that that collection satisfies Property~(1) and Property~(2) in the hypotheses of Theorem~\ref{t:meqstrmeq}.

For the second bullet point, Property~(1) follows from basic graph algebra results stating a graph whose $C^*$-algebra has a largest ideal that is also AF must satisfy Condition~(K), and Property~(2) follows from \cite[Theorem~4.6]{semt_classgraphalg}.

For the third bullet point, Property~(1) follows from basic graph algebra results stating a graph whose $C^*$-algebra has smallest nonzero ideal that is purely infinite and whose corresponding quotient is AF must satisfy Condition~(K), and Property~(2) follows from \cite[Corollary~6.4]{segrer:ccfis}.

For the fourth bullet point, the operation of source removal (see \cite[Proposition~3.1]{as:geo}) allows us to assume that without loss of generality that the graphs each have no sources.  Thus the graphs are finite graphs with no sinks or sources satisfying Condition~(K), and $C^*(E)$ and $C^*(F)$ are each Cuntz-Krieger algebras of matrices satisfying Condition~(II).  Property~(1) then holds by hypothesis, and Property~(2) follows from Restorff's results in \cite{gr:ckalg}, which show that Cuntz-Krieger algebras of matrices satisfying Condition~(II) are classified up to stable isomorphism by their ideal-related topological $K$-theory. 

For the fifth bullet point, Property~(1) holds because all amplified graphs satisfy Condition~(K), and Property~(2) follows from \cite[Theorem~5.7]{seaser:amplified}.

For the sixth bullet point, Property~(1) holds because any graph $C^*$-algebra with a primitive ideal space that is an accordion space has a finite number of ideals and hence must come from a graph satisfying Condition~(K), and Property~(2) follows from \cite{rasmus} and \cite[Theorem~1.3]{rasmusmanuel}. 

For the seventh bullet point, Property~(1) holds because any simple graph $C^*$-algebra must come from a graph that satisfies Condition~(K), and Property~(2) follows the fact that any simple graph $C^*$-algebra is either purely infinite or AF, and hence classified up to stable isomorphism by either Elliott's theorem or the Kirchberg-Phillips classification theorem.  (We also observe that ideal-related topological $K$-theory reduces to ordered $K$-theory in the simple case.)
\end{proof}

\begin{remar}
The result in the last bullet point of Corollary~\ref{many-classifications-cor} extends \cite[Corollary~9.16]{gamt:isomorita} to non-row-finite graphs.
\end{remar}

\section{The Isomorphism Conjecture for graph algebras} \label{Iso-Con-sec}

In this section we consider the isomorphism conjecture for graph algebras.  As in the previous section, we consider ideal-related $K$-theory, however, now we must also keep track of the position of the class of the unit in the $K_0$-group.  We prove that if the ideal-related  algebraic $K$-theories of two unital Leavitt path algebras of graphs satisfying Condition~(K) are isomorphic via an isomorphism taking the class of the unit to the class of the unit, then there is an isomorphism between the ideal-related topological $K$-theories of the graph $C^*$-algebras taking the class of the unit to the class of the unit.  This allows us to prove that the  isomorphism conjecture for graph algebras holds for any class of graphs whose $C^*$-algebras are classified up to isomorphism by their ideal-related topological $K$-theory plus the position of the class of the unit.  By applying existing classification theorems, we confirm the isomorphism conjecture for a number of specific classes of graphs in the corollaries at the end of this section.

\begin{lemma}\label{l:cornerlocalunit}
Let $R$ be a ring and let $e \in R$ be an idempotent such that $R e R = R$.  Let $I$ be an ideal of $R$.  If $I$ has a countable approximate unit consisting of idempotents, then the ideal of $R$ generated by $e I e$ is equal to $I$.  Moreover, if every ideal of $R$ has a countable approximate unit consisting of idempotents, then $$I \mapsto e I e$$ is a lattice isomorphism from $\ilat ( R )$ to $\ilat ( e R e )$ with inverse given by $$I \mapsto \text{the ideal in $R$ generated by $I$}.$$
\end{lemma}

\begin{proof}
Suppose  $\{ u_{n} \}_{ n = 1}^{ \infty }$ is a countable approximate unit consisting of idempotents for $I$.  We will first show that $I e I = I$, where $I e I$ is the ideal of $I$ given by
\begin{align*}
I e I := \setof{ e r + s e + \sum_{ i = 1}^{m} r_{i} e s_{i} }{ r, s, r_{i}, s_{i} \in I, m \in \N }.
\end{align*}
It is clear that $I e I \subseteq I$.  

Let $x \in I$.  Since $R e R = R$, we have that $x \in R e R$.  Since $u_{n} \in I$ for each $n \in \N$, we have that $u_{n} x u_{n} \in I e I$ for each $n \in \N$.  Since $\{ u_{n} \}_{ n = 1}^{ \infty }$ is an approximate unit of $I$, there exists $n \in \N$ such that $u_{n} x u_{n} = x$.  Therefore, $x = u_{n} x u_{n} \in I e I$.  Thus $I \subseteq IeI$, and $IeI = I$.

Let $J$ be the ideal of $R$ generated by $e I e$.  Since $eIe \subseteq I$ and $I$ is an ideal of $R$, we have that  $J \subseteq I$.  Let $x \in I$.  Since $IeI=I$, there exist $m \in \N$, $r, s, r_{i}, s_{i} \in I$ for $i = 1, 2, \dots, m$ such that 
\begin{align*}
x= e r + s e + \sum_{ i = 1}^{m} r_{i} e s_{i}.
\end{align*}
Since $er \in I$, there exists $n_{1} \in \N$ such that $u_{ n_{1} } er = er$.  Hence
$$ er = eer = e u_{n_1} e r \in (eIe)R \subseteq J.$$
Likewise, since $se \in I$, there exists $n_{2} \in \N$ such that $se u_{n_{2}} = se$.  Hence, 
$$ se = see = seu_{n_2} e \in R (eIe) \subseteq J.$$
Finally, for each $1 \leq i \leq m$, we have $e s_{i} \in I$ and hence there exists $k_{i} \in \N$ such that $u_{k_{i}} e s_{i} = e s_{i}$.  Hence
$$r_i e s_i = r_i e e s_i = r_i e u_{k_{i}} e s_{i} \in R (eIe) R \subseteq J.$$
Thus we have
\begin{align*}
x= e r + s e + \sum_{ i = 1}^{m} r_{i} e s_{i} \in J
\end{align*}
and we may conclude that $I \subseteq J$ and $I=J$.

Moreover, if every ideal of $R$ has a countable approximate unit consisting of idempotents, then it is straightforward to see that the map $\mu : \ilat ( R ) \to \ilat ( e R e )$ defined by $\mu (I) = eIe$ has inverse 
\begin{align*}
\rho : \ilat (eRe) \to \ilat (R)
\end{align*}
defined by $\rho (J)$ equals the ideal in $R$ generated by $J$; indeed, the fact that $\rho \circ \mu = \textrm{id}$ is the result of the first part of this proof, and the fact that $\mu \circ \rho = \textrm{id}$ is straightforward to verify.

In addition, $\mu$ preserves inclusion, and since any bijection between lattices that preserves inclusion is a lattice isomorphism, we may conclude that $\mu$ is a lattice isomorphism.
\end{proof}

\begin{lemma}\label{l:fcindmap}
Let $R$ be a ring containing a countable approximate unit consisting of idempotents.  If $e$ is an idempotent of $R$ such that $ReR = R$, then the homomorphism $\ftn{ \kalg_{n} ( \iota ) }{ \kalg_{n} ( e R e ) }{ \kalg_{n} (  R ) }$ induced by the inclusion $\ftn{ \iota } { e R e } { R }$ is an isomorphism for all $n \in \Z$.
\end{lemma}

\begin{proof}
If $R$ is unital, this is a well-known fact (see \cite[Lemma~2.7]{abc:kthyleavitt}).  If $R$ is not unital, we shall show that we can obtain the result from the unital case by a direct limit argument.

Let $\{ e_{k} \}_{ k = 1}^{ \infty }$ be a countable approximate unit for $R$ consisting of idempotents.  Without loss of generality, we may assume that $e_k e = ee_k = e$ for all $k \in \N$.  Also, $R = \bigcup_{k=1}^\infty e_k Re_k$ (see Remark~\ref{sigma-unital-dir-lim-rem}).  

Since $eRe = e_k(eRe)e_k$, we will define $\phi_k : eRe \to e_k Re_k$ to be the inclusion map with $\phi_k(x) := x$.  We see that $eRe$ is a full corner of $e_k R e_k$ because
$$ (e_kRe_k) e (e_kRe_k) = e_k R (e_kee_k) R e_k = e_k (ReR) e_k = e_k(R) e_k.$$
Thus, by  \cite[Lemma~2.7]{abc:kthyleavitt} for every $n \in \Z$, 
\begin{align*}
\kalg_n(\phi_k) : \kalg_n (eRe) \to \kalg_n (e_kRe_k)
\end{align*}
is an isomorphism.  If $\iota_{k,k+1} : e_k R e_k \to e_{k+1} R e_{k+1}$ is the inclusion map, we see that $\iota_{k, k+1} \circ \phi_k = \phi_{k+1}$.  Hence $$\kalg_n (\iota_{k, k+1} )\circ \kalg_n( \phi_k) =\kalg_n (\phi_{k+1}).$$
By the universal property of the direct limit the induced map from $\kalg_n (eRe)$ to $\varinjlim (\kalg_n (e_kRe_k), \kalg_n(\iota_{k,k+1})) = \kalg_n (R)$ is an isomorphism.  But this induced map is precisely $\kalg_n (\iota)$.
\end{proof}

\begin{lemma}\label{l:inducealgkthy}
Let $R$ and $S$ be rings such that every subquotient of $R$ satisfies excision in algebraic $K$-theory and every subquotient of $S$ satisfies excision in algebraic $K$-theory, and let $\ftn{ \phi }{ R }{ S }$ be a homomorphism.  Suppose the following two conditions are satisfied:
\begin{itemize}
\item[(1)] the map $\ftn{ \beta }{ \ilat ( R ) }{ \ilat ( S ) }$ defined by $$\beta( I ) := \text{the ideal in $S$ generated by $\phi ( I )$} $$ is a lattice isomorphism, and 
\item[(2)] for each $I_{1} , I_{2} \in \ilat ( R )$ with $I_{1} \subseteq I_{2}$, the homomorphism $\ftn{ \phi_{ I_{2} / I_{1} } }{ I_{2} /  I_{1}  }{ \beta(I_{2}) / \beta(I_{1}) }$ defined by $$  \phi_{ I_{2} / I_{1} }  ( x + I_{1}  ) := \phi ( x ) + \beta(I_{1})$$ induces an isomorphism $$\ftn{ \kalg_{n} ( \phi_{ I_{2} / I_{1} } ) }{ \kalg_{n} ( I_{2} /  I_{1}  )  }{ \kalg_{n} ( \beta(I_{2}) / \beta(I_{1}) ) }$$ for each $n \in \Z$ with $\kalg_{0} ( \phi_{ I_{2} / I_{1} } )$ also being an order isomorphism.
\end{itemize}
Then the collection
\begin{align*}
\setof{  \kalg_{n} ( \phi_{ I_{2} / I_{1} } )} { I_{1}, I_{2} \in \ilat ( R ), \ I_{1} \subseteq I_{2}, \text{ and } n \in \Z },
\end{align*}
together with the lattice isomorphism $\beta$, is an isomorphism of ideal-related  algebraic $K$-theory, and $\oikalg ( R ) \cong  \oikalg ( S )$.
\end{lemma}

\begin{proof}
Note that for $I_{1} , I_{2} , I_{3} \in \ilat( R )$ with $I_{1} \subseteq I_{2} \subseteq I_{3}$, we have a commutative diagram

$$
\xymatrix{
0 \ar[r] & I_{2} / I_{1} \ar[r] \ar[d]^{ \phi_{ I_{2} / I_{1} } } & I_{3}  /  I_{1}  \ar[r] \ar[d]^{ \phi_{ I_{3} / I_{1} } } &  I_{3}  /  I_{2}  \ar[r]  \ar[d]^{ \phi_{ I_{3} / I_{2} } } & 0 \\
0 \ar[r] & \beta(I_{2}) / \beta(I_{1}) \ar[r] & \beta(I_{3}) / \beta(I_{1}) \ar[r] & \beta(I_{3}) / \beta(I_{2}) \ar[r] & 0 
}
$$
Applying the functor $\kalg_n$ we obtain the following commutative diagram:

$ $

\scalebox{.75}{$$
\xymatrix{
{\kalg_{n}(I_{2}/ I_{1} ) } \ar[r] \ar[d]^{ \kalg_{n} ( \phi_{ I_{2} / I_{1} } ) }&{ \kalg_{n}(I_{3} / I_{1})}  \ar[r] \ar[d]^{ \kalg_{n} ( \phi_{ I_{3} / I_{1} } )} & {\kalg_{n}(I_{3}/I_{2}) }\ar[r] \ar[d]^{ \kalg_{n} ( \phi_{ I_{3} / I_{2} } ) } & {\kalg_{n-1} (I_{2}/ I_{1})} \ar[d]^{ \kalg_{n-1} ( \phi_{ I_{2} / I_{1} } ) } \\
{ \kalg_{n}( \beta( I_{2} ) / \beta( I_{1} ) ) } \ar[r] &{ \kalg_{n}( \beta( I_{3} ) / \beta( I_{1} ) ) }\ar[r] &{ \kalg_{n}( \beta( I_{3} ) / \beta( I_{2} ) ) }\ar[r]  & { \kalg_{n-1} ( \beta( I_{2} ) / \beta( I_{1} ) )}
}
$$
}

$ $

$ $

\noindent Thus $\beta$ and the collection of $\kalg_{n} ( \phi_{ I_{2} / I_{1} } )$ for all $n \in \Z$ and $I_{1}, I_{2} \in \ilat ( R )$ with $I_{1} \subseteq I_{2}$ give an isomorphism of ideal-related algebraic $K$-theories, and $\oikalg ( R ) \cong  \okalg ( S )$. 
\end{proof}

We have a similar result for $C^*$-algebras.

\begin{lemma}\label{l:inducetopkthy}
Let $\mathfrak{A}$ and $\mathfrak{B}$ be $C^{*}$-algebras, and let $\ftn{ \phi }{ \mathfrak{A} }{ \mathfrak{B} }$ be a homomorphism.   Suppose
\begin{itemize}
\item[(1)] the map $\ftn{ \beta }{ \ilatcl ( \mathfrak{A} ) }{ \ilatcl ( \mathfrak{B} ) }$ defined by $$\beta( \mathfrak{I} ) := \text{the ideal in $\mathfrak{B}$ generated by $\phi ( \mathfrak{I} )$} $$ is a lattice isomorphism; and 

\item[(2)] for each $\mathfrak{I}_{1} , \mathfrak{I}_{2} \in \ilat ( \mathfrak{A} )$ with $\mathfrak{I}_{1} \subseteq \mathfrak{I}_{2}$ the homomorphism $\ftn{ \phi_{ \mathfrak{I}_{2} / \mathfrak{I}_{1} } }{ \mathfrak{I}_{2} /  \mathfrak{I}_{1}  }{ \beta(\mathfrak{I}_{2}) / \beta(\mathfrak{I}_{1}) }$ defined by 
$$ \phi_{ \mathfrak{I}_{2} / \mathfrak{I}_{1} } (x + \mathfrak{I}_{1})  := \phi ( x ) + \beta(\mathfrak{I}_{1})$$
induces an isomorphism 
$$\ftn{ \ktop_{n} ( \phi_{ \mathfrak{I}_{2} / \mathfrak{I}_{1} } ) }{ \ktop_{n} ( \mathfrak{I}_{2} /  \mathfrak{I}_{1}  )  }{ \ktop_{n} ( \beta(\mathfrak{I}_{2}) / \beta(\mathfrak{I}_{1}) ) }$$ for each $n =0,1$ with $\ktop_{0} ( \phi_{ \mathfrak{I}_{2} / \mathfrak{I}_{1} } )$ also being an order isomorphism.
\end{itemize}  
Then the collection
\begin{align*}
\setof{ \ktop_{n} ( \phi_{ \mathfrak{I}_{2} / \mathfrak{I}_{1} } )} { \mathfrak{I}_{1}, \mathfrak{I}_{2} \in \ilatcl ( \mathfrak{A} ), \ \mathfrak{I}_{1} \subseteq \mathfrak{I}_{2}\ , \text{ and } n=0,1}
\end{align*}
with the lattice isomorphism $\beta$ gives an isomorphism of ideal-related topological $K$-theories, and
$\oiktop ( \mathfrak{A} ) \cong  \oiktop ( \mathfrak{B} )$.
\end{lemma}

\begin{proof}
Note that every $C^{*}$-algebra satisfies excision in topological $K$-theory.  Using the same argument as in the proof of Lemma~\ref{l:inducealgkthy} we get the desired result. 
\end{proof}

\begin{defin}
We establish notation and terminology for the situations in Lemma~\ref{l:inducealgkthy} and Lemma~\ref{l:inducetopkthy}.

Let $R$ and $S$ be rings such that every subquotient of $R$ satisfies excision in algebraic $K$-theory and every subquotient of $S$ satisfies excision in algebraic $K$-theory.  If $\ftn{ \phi }{ R }{ S }$ is a homomorphism satisfying the assumptions of Lemma~\ref{l:inducealgkthy}, we write $$[\phi] := \left( \left\{  \kalg_{n} ( \phi_{ I_{2} / I_{1} } ) : I_1, I_2 \in \ilat(R) \text{ with } I_1 \subseteq I_2, n\in \Z \right\}, \beta \right)$$
and say that $\phi$ induces an isomorphism $\ftn{ [ \phi ] }{ \oikalg ( R ) }{ \oikalg ( S ) }$. 

Let $\mathfrak{A}$ and $\mathfrak{B}$ be $C^{*}$-algebras.  If $\ftn{ \phi }{ \mathfrak{A} }{ \mathfrak{B} }$ is a homomorphism satisfying the assumptions in Lemma~\ref{l:inducetopkthy}, then we write 
$$[\phi] := \left( \left\{  \ktop_{n} ( \phi_{ \mathfrak{I}_{2} / \mathfrak{I}_{1} } ) : \mathfrak{I}_1, \mathfrak{I}_2 \in \ilatcl(\mathfrak{A}) \text{ with } \mathfrak{I}_1 \subseteq \mathfrak{I}_2, n =0,1 \right\}, \beta \right)$$
and say that $\phi$ induces an isomorphism $\ftn{ [ \phi ] }{ \oiktop ( \mathfrak{A} ) }{ \oiktop ( \mathfrak{B} ) }$.
\end{defin}

\begin{lemma}\label{l:kthylocalunits}
Let $R$ be a ring with a countable approximate unit consisting of idempotents.  If $x \in \kalg_{0} ( R )$, then there exist idempotents $e \in \mathsf{M}_{k} ( R )$ and $f \in \mathsf{M}_{k} ( R )$ such that $x = [ e ] - [ f ]$.
\end{lemma}

\begin{proof}
If $R$ is unital, then this is clear by the definition of $\kalg_{0} ( R )$.  Suppose $R$ is a ring with a countable approximate unit consisting of idempotents $\{ e_n \}_{n=1}^\infty$.  Let $\ftn{ \iota_{n, n+1} }{ e_{n} R e_{n} }{ e_{n+1} R e_{n+1} }$ and $\ftn{ \iota_{n} }{ e_{n} R e_{n} }{ R }$ be the inclusion homomorphisms.  Then $R \cong \varinjlim ( e_{n} R e_{n} , \iota_{n,n+1} )$, and by continuity of $K$-theory
\begin{align*}
\kalg_{0} ( R ) \cong \varinjlim ( \kalg_{0} ( e_{n} R e_{n} ) , \kalg_{0} ( \iota_{n, n+1}  ) ).
\end{align*}
Let $\ftn{ \alpha }{  \varinjlim ( \kalg_{0} ( e_{n} R e_{n} ) , \kalg_{0} ( \iota_{n , n+1 } ) ) }{ \kalg_{0} ( R ) } $ denote this isomorphism and let $\ftn{ \psi_{n} }{ \kalg_{0} ( e_{n} R e_{n} ) }{ \varinjlim ( \kalg_{0} ( e_{n} R e_{n} ) , \kalg_{0} ( \iota_{n , n+1 } ) ) }$ be the natural map to the inductive limit.  Then for each $n \in \N$ we have $\alpha \circ \psi_{n} = \kalg_{0} ( \iota_{n} )$. 

Let $x \in \kalg_{0} ( R )$.  Then there exists $y \in\varinjlim ( \kalg_{0} ( e_{n} R e_{n} ) , \kalg_{0} ( \iota_{n , n+1 } ) ) $ such that $\alpha ( y ) = x$.  Hence, there exist $n \in \N$ and $y_{n} \in \kalg_{0} ( e_{n} R e_{n} )$ such that $\psi_{n} ( y_{n} ) = y$.  Since $e_{n} R e_{n}$ is a unital ring, $y_{n} = [ a _{n} ] - [ b_{n} ]$ for some idempotents $a_{n}, b_{n} \in \mathsf{M}_{k} ( e_{n} R e_{n} )$.  Hence, if we let $e = \iota_{n} ( a_{n} )$ and $f = \iota_{n} ( b_{n} )$, then
\begin{align*}
x &= \alpha ( y ) = ( \alpha \circ \psi_{n} ) ( y_{n} ) = \kalg_{0} ( \iota_{n} ) ( [a_{n}] - [ b_{n} ] ) = [ e ] - [ f ].
\end{align*}   
\end{proof}

\begin{lemma}\label{l:projfullcorner}
Let $R$ be a ring and let $e$ be an idempotent such that $R e R = R$.  Let $I_{1}, I_{2} \in \ilat ( R )$ with $I_{1} \subseteq I_{2}$.  Then for every idempotent $p \in \mathsf{M}_{n} ( I_{2} / I_{1} )$, there exists an idempotent $q \in \mathsf{M}_{k} ( e I_{2} e / e I_{1} e )$ such that $[ p ] = [ q ]$ in $\kalg_{0} ( I_{2} / I_{1} )$.
\end{lemma}

\begin{proof}
Let $I_{1}, I_{2} \in \ilat ( R )$ such that $I_{1} \subseteq I_{2}$ and let $p$ be an idempotent in $\mathsf{M}_{n} ( I_{2} / I_{1} )$.  Since $R e R = R$, we have $( R / I_{1} ) \overline{e} ( R / I_{1}) = R/ I_1$, where $\overline{e} = e+I_{1}  \in R / I_{1}$.  For each $n \in \N$, define $\overline{e}_{n} = \diag ( \overline{e} , \dots, \overline{e}  ) \in \mathsf{M}_{n} ( R / I )$.

Since $( R / I_{1} ) \overline{e} ( R / I_{1}) = R /I_1$, we have $\mathsf{M}_{n} ( R /I_{1} ) \overline{e}_{n} \mathsf{M}_{n} ( R / I_{1} ) = \mathsf{M}_{n} ( R / I_{1} )$ for all $n \in \N$.  Hence, $p \in \mathsf{M}_{n} ( R /I_{1} ) \overline{e}_{n} \mathsf{M}_{n} ( R / I_{1} )$ and  there exist $$x_{1}, \dots, x_{m} , y_{1}, \dots, y_{m} \in \mathsf{M}_{n} ( R / I_{1} )$$ and $\ell \in \Z$ with $\ell \geq 0$ such that 
\begin{align*}
p = \ell \overline{e}_{n} + x_{1} \overline{e}_{n} + \overline{e}_{n} y_{1} + \sum_{ i = 2}^{m} x_{i} \overline{e}_{n } y_{i}. 
\end{align*}   
Since $p$ is an idempotent, 
\begin{align*}
p = \ell p \overline{e}_{n} p + p x_{1} \overline{e}_{n} p + p \overline{ e}_{n} y_{1} p + \sum_{ i = 2}^{m} p x_{i} \overline{e}_{n} y_{i} p
\end{align*}
Then, since $p \in M_n(I_2/I_1)$, by grouping terms we may write  
\begin{align*}
p = \sum_{ i = 1}^{\ell + m +1 } s_{i} \overline{e}_{n} t_{i}.
\end{align*}
for $s_i, t_i \in M_n(I_2/I_1)$. Set 
\begin{align*}
s = 
\begin{bmatrix}
s_{1}& s_{2} & \dots & s_{\ell+m+1} \\
0 	& 0 & \dots & 0 \\
\vdots & \vdots &  & \vdots \\
0 & 0 & \dots & 0
\end{bmatrix}
\quad \text{and} \quad 
t = \begin{bmatrix}
t_{1}  & 0 & \dots & 0  \\
t_{2}  	& 0 & \dots & 0 \\
\vdots & \vdots &  & \vdots \\
t_{\ell+m+1}  & 0 & \dots & 0
\end{bmatrix}
\end{align*}
Then $s, t \in \mathsf{M}_{n(\ell+ m+1 ) } ( I_{2} / I_{1} )$ such that 
\begin{align*}
\begin{bmatrix}
p & 0 & \dots & 0 \\
0 & 0 & \dots & 0 \\
\vdots & \vdots  & \vdots \\
0 & 0 & \dots & 0
\end{bmatrix} = s \overline{e}_{n(\ell + m+1)} t
\end{align*}
Let $z = s \overline{e}_{n(\ell+m+1)}$ and $w = \overline{e}_{n(\ell+m+1)} t$.  Then $zw = \diag (p,0, \ldots 0)$ is an idempotent.  While we do not know if $wz$ is an idempotent, we do see that $(wz)^2 = w(zw)z = w(zw)^3z = (wz)^4$, so that $(wz)^2$ is an idempotent.  Moreover, $(zw)^2 = zwzw \sim wzwz = (wz)^2$ so that
$$[p] = [p^2] = [(zw)^2] = [(wz)^2]$$ in $\kalg_{0} ( I_{2} / I_{1} )$.  Since $(wz)^2 =  (\overline{e}_{n(\ell+m+1)} t s \overline{e}_{n(\ell+m+1)})^2$ we have that $$(wz)^2 \in  \overline{e}_{n( \ell+m+1) }  \mathsf{M}_{n ( \ell + m +1) } (  I_{2} / I_{1} ) \overline{e}_{n( \ell+m+1) } = \mathsf{M}_{ n ( \ell + m +1 ) } ( \overline{e} (I_{2} / I_{1}) \overline{e} ).$$ 
Since the map $\ftn{ \psi }{ e I_{2} e / e I_{1} e }{ \overline{e} (I_{2} / I_{1}) \overline{e} }$ defined by $\psi ( x + e I_{1} e ) = \overline{e} ( x + I_{1} ) \overline{e}$ is an isomorphism, the lemma follows.
\end{proof}

\begin{propo}\label{p:algkthyfullcorner}
Let $R$ be a ring and let $e \in R$ be an idempotent such that $R e R = R$.  Suppose every ideal of $R$ has a countable approximate unit consisting of idempotents, and every ideal of $eRe$ has a countable approximate unit consisting of idempotents.  Then the inclusion $\ftn{ \iota }{ e R e }{ R }$ induces an isomorphism from $\oikalg ( eRe )$ to $\oikalg ( R )$.  
\end{propo}

\begin{proof}
First note that since every ideal $I$ of $R$ has a countable approximate unit consisting of idempotents, Lemma~\ref{l:excisionlocalunit} implies that every subquotient of $R$ satisfies excision in algebraic $K$-theory.  By Lemma~\ref{l:cornerlocalunit}, the map $\ftn{ \beta }{ \ilat ( eRe ) }{ \ilat ( R  ) }$ given by $$\beta ( I ) =  \text{the ideal in $R$ generated by $I$}$$ is a lattice isomorphism.  We will now show that for each $I_{1}, I_{2} \in \ilat ( R )$ with $I_{1} \subseteq I_{2}$, 
\begin{align*}
\ftn{ \kalg_{n} ( \iota_{ I_{2} / I_{1} } ) }{ \kalg_{n} ( I_{2} / I_{1} ) }{ \kalg_{n} ( \beta(I_{2}) / \beta(I_{1}) ) }
\end{align*}
is an isomorphism for each $n \in \Z$ with $\kalg_0 ( \iota_{ I_{2} / I_{1} } )$ also being an order isomorphism.  

\medskip

\noindent \emph{Case 1:}  Suppose $I_2 = eRe$ and $I_1$ is any ideal of $eRe$. Let $\overline{e} := e + \beta (I_1) \in R / \beta(I_1)$.  Then $I_{1} = e \beta(I_1) e$. We see that $eRe / I_1 \cong \overline{e} ( R / \beta(I_1) ) \overline{e}$, and composing this isomorphism with the inclusion $\overline{e} ( R / \beta(I_1) ) \overline{e} \hookrightarrow R / \beta(I_1)$ gives $\iota_{I_2/I_1}$.  That each $\kalg_n ( \iota_{ I_{2} / I_{1} } )$ is an isomorphism follows from Lemma~\ref{l:fcindmap}, and that $\kalg_0 ( \iota_{ I_{2} / I_{1} } )$ is also an order isomorphism follows from Lemma~\ref{l:projfullcorner}.

\medskip

\noindent \emph{Case 2:} Suppose $I_{1} = 0$ and $I_{2}$ is any ideal of $eRe$.
Note that the diagram
\begin{align*}
\xymatrix{
0 \ar[r] & I_{2} \ar[r] \ar[d]^{ \iota_{ I_{2} / 0 } } & e R e \ar[r] \ar[d]^{ \iota_{ eRe/ 0 } } & e R e / I_2 \ar[r]  \ar[d]^{ \iota_{ eRe/I_{2} } }& 0 \\
0 \ar[r] & \beta(I_{2})  \ar[r] &  R  \ar[r] &  R  /  \beta(I_{2}) \ar[r] & 0 
}
\end{align*}  
commutes, and induces the following commutative diagram
\begin{align*}
\scalebox{.8}{ 
\xymatrix{
\kalg_{n+1} ( e R e ) \ar[r] \ar[d]^{ \kalg_{n+1} ( \iota_{ eRe / 0 } ) } &  \kalg_{n+1} ( e R e / I_{2} ) \ar[d]^{ \kalg_{n} ( \iota_{eRe / I_{2} } ) } \ar[r] & \kalg_{n} ( I_{2} ) \ar[r] \ar[d]^{ \kalg_{n} (\iota_{ I_{2}/ 0 })  } & \kalg_{n} ( e R e ) \ar[r] \ar[d]^{ \kalg_{n} ( \iota_{ eRe/ 0 } ) } & \kalg_{n} ( e R e /I_{2} )   \ar[d]^{ \kalg_{n} (\iota_{ eRe/ I_{2} } ) } \\
 \kalg_{n+1} ( R ) \ar[r] & \kalg_{n+1} ( R / \beta(I_{2}) ) \ar[r] & \kalg_{n} ( \beta(I_{2}))  \ar[r] &  \kalg_{n} ( R ) \ar[r] &  \kalg_{n} ( R  /  \beta(I_{2}) )  
}}
\end{align*}
in $K$-theory.
Since $e I e$ and $I$ satisfy excision in algebraic $K$-theory, the rows are exact.  Hence, by the Five Lemma and Case 1, $\kalg_{n} ( \iota_{ I_{2} /  0 } )$ is an isomorphism.  In addition, by Lemma~\ref{l:projfullcorner} $\kalg_0 ( \iota_{ I_{2} / 0 } )$ is also an order isomorphism.

\medskip

\noindent \emph{Case 3:}  Suppose $I_1$ and $I_2$ are ideals in $eRe$ with $I_1 \subseteq I_2$.  Then the diagram 
\begin{align*}
\xymatrix{
0 \ar[r] & I_{1} \ar[r] \ar[d]^{ \iota_{ I_{1}/0}} & I_{2} \ar[r] \ar[d]^{ \iota_{I_{2}/0}} & I_{2} / I_{1} \ar[r] \ar[d]^{ \iota_{ I_{2}/I_{1} } }& 0 \\
0 \ar[r] & \beta(I_{1}) \ar[r] & \beta(I_{2}) \ar[r] & \beta(I_{2}) / \beta(I_{1}) \ar[r] & 0
 }
\end{align*} 
is commutative and induces the following commutative diagram
\begin{align*}
\scalebox{.75}{
\xymatrix{
\kalg_{n} ( I_{1} ) \ar[r] \ar[d]^{ \kalg_{n} ( \iota_{ I_{1}/0} ) } & \kalg_{n} (I_{2}) \ar[r] \ar[d]^{ \kalg_{n} ( \iota_{I_{2}/0} ) } & \kalg_{n} ( I_{2} / I_{1} ) \ar[r] \ar[d]^{ \kalg_{n} ( \iota_{I_{2}/I_{1} } ) } & \kalg_{n-1} ( I_{1} ) \ar[d]^{ \kalg_{n-1} ( \iota_{ I_{1}/0 } ) } \ar[r] & \kalg_{n-1} ( I_{2} ) \ar[d]^{ \kalg_{n-1} ( \iota_{I_{2} / 0 } ) } \\
\kalg_{n} ( \beta(I_{1}) ) \ar[r] & \kalg_{n} ( \beta(I_{2}) ) \ar[r] & \kalg_{n} ( \beta(I_{2})/ \beta(I_{1}) ) \ar[r] & \kalg_{n-1} ( \beta(I_{1})) \ar[r] & \kalg_{n-1} ( \beta(I_{2}) ) 
}}
\end{align*} 
Since $I_{1}$ and $\beta(I_1)$ satisfy excision in algebraic $K$-theory, the rows are exact.  Hence, by the Five Lemma and by Case 2, we have that $\kalg_{n} ( \iota_{ I_{2}/ I_{1} } )$ is an isomorphism.  In addition, by Lemma~\ref{l:projfullcorner} $\kalg_0 ( \iota_{ I_{2} / I_1 } )$ is also an order isomorphism.
\end{proof}

\begin{propo}\label{p:topkthyfullcorner}
Let $\mathfrak{A}$ be a separable $C^{*}$-algebra and let $p \in \mathfrak{A}$ be a projection such that $\overline{ \mathfrak{A} p \mathfrak{A} } = \mathfrak{A}$.  Then the inclusion $\ftn{ \iota }{ p \mathfrak{A} p }{ \mathfrak{A} }$ induces an isomorphism from $\oiktop ( p\mathfrak{A}p )$ to $\oiktop (\mathfrak{A} )$.  
\end{propo}

\begin{proof}
Every $C^{*}$-algebra satisfies excision in topological $K$-theory and by L.~G.~Brown \cite{heralgs}, the natural embedding of any hereditary sub-$C^{*}$-algebra of a separable $C^{*}$-algebra $\mathfrak{B}$ that is not contained in any proper closed two-sided ideal of $\mathfrak{B}$ induces an order isomorphism in topological $K$-theory.  Also, note that every subquotient $p \mathfrak{I}_{2} p / p\mathfrak{I}_{1} p$ of $p \mathfrak{A}p$ is isomorphic to a hereditary sub-$C^{*}$-algebra of $\mathfrak{I}_{2} / \mathfrak{I}_{1}$ that is not contained in any proper closed two-sided ideal of $\mathfrak{I}_{2} / \mathfrak{I}_{1}$.  Using these facts and using the same argument as in Proposition~\ref{p:algkthyfullcorner}, we get the desired result. 
\end{proof}

\begin{lemma} \label{corner-lemma}
Let $E$ be a graph with finitely many vertices and let $F$ be a desingularization of $E$.  Let $\{ s_e, p_v : e \in F^1, v \in F^0 \} \subseteq L_\C(F) \subseteq C^*(F)$ be a generating Cuntz-Krieger $F$-family.  Let $p = \sum_{ v \in E^{0} } p_{v} \in L_\C(F)$. Then $pL_\C(F)p$ is a full corner of $L_\C(F)$, $pC^*(F)p$ is a full corner of $C^*(F)$, and there exist isomorphisms $\ftn{ \phi }{ L_{\C} ( E ) }{ p L_{\C} ( F ) p }$ and $\ftn{ \psi }{ C^{*} ( E ) }{ p C^{*} ( F ) p }$ such that the diagram 
\begin{align*}
\xymatrix{
L_{\C} ( E ) \ar[r]^{ \phi } \ar[d]^{ \iota_{E } } & p L_{\C} ( F ) p \ar@{^{(}->}[r] \ar[d]^{ \iota_{F} } & L_{\C} ( F ) \ar[d]^{ \iota_{F} } \\
C^{*} ( E ) \ar[r]_{ \psi } & p C^{*} ( F ) p \ar@{^{(}->}[r] & C^{*} ( F )
}
\end{align*}
commutes.
\end{lemma}

\begin{theor}\label{t:induceisounit}
Let $E_{1}$ and $E_{2}$ be graphs that each have finitely many vertices and satisfy Condition~(K).  If there exists an isomorphism $$\ftn{ \alpha }{ \oikalg  ( L_{\C} ( E_{1} ) ) }{ \oikalg ( L_{\C} ( E_{2} ) ) }$$ with $\alpha^{ 0 , L_{\C} ( E_{1} ) } ([ 1_{ L_{\C} ( E_{1} ) } ]) = [ 1_{ L_{\C} ( E_{2} ) } ]$, then there exists an isomorphism $$\ftn{ \beta }{ \oiktop ( C^{*} ( E_{1} ) ) }{ \oiktop ( C^{*} ( E_{2} ) ) }$$ with $\beta^{ 0, C^{*} (E_{1}) } ([ 1_{ C^{*}( E_{1} ) } ])= [ 1_{ C^{*} ( E_{2} ) } ]$.
\end{theor}

\begin{proof}
For $i=1,2$, let $F_{i}$ be a desingularization of $E_{i}$.  By Lemma~\ref{corner-lemma}, for each $i=1,2$ there exists an idempotent $p_i \in L_\C(F_i)$ such that $p_iL_\C(F_i)p_i$ is a full corner of $L_\C(F_i)$, $p_iC^*(F_i)p_i$ is a full corner of $C^*(F_i)$, and there exist isomorphisms $\ftn{ \phi_i }{ L_{\C} ( E_i ) }{ p_i L_{\C} ( F_i ) p_i }$ and $\ftn{ \psi_i }{ C^{*} ( E_i ) }{ p_i C^{*} ( F_i ) p_i }$ such that the diagram 
\begin{align*}
\xymatrix{
L_{\C} ( E_{i} ) \ar[r]^-{\phi_{i}} \ar[d]_{ \iota_{E_{i} } } & p_{i} L_{\C} ( F_{i} ) p_{i} \ar[d]^{ \iota_{ F_{i} } } \ar@{^{(}->}[r]^-{ \iota^{ \mathrm{alg} } } & L_{\C} ( F_{i} ) \ar[d]^{ \iota_{ F_{i} } }  \\
C^{*} ( E_{i} ) \ar[r]_-{ \psi_{i} } & p_{i} C^{*} ( F_{i} ) p_{i} \ar@{^{(}->}[r]_-{ \iota^{\mathrm{top}} } & C^{*} ( F_{i} )
}
\end{align*}
commutes.

By hypothesis there is an isomorphism $$\ftn{ \alpha }{ \oikalg ( L_{\C} ( E_{1} ) ) }{ \oikalg ( L_{\C} ( E_{2} ) ) }$$ with  $\alpha^{ 0 , L_{\C} ( E_{1} ) } ([ 1_{ L_{\C} ( E_{1} ) } ]) = [ 1_{ L_{\C} ( E_{2} ) } ]$.  Since $p_{i} L_{\C} ( F_{i} ) p_{i} \cong L_{\C} ( E_{i} )$, there exists an isomorphism $\ftn{ \widetilde{\alpha} }{ \oikalg ( p_{1} L_{\C} ( F_{1} ) p_{1} ) }{ \oikalg ( p_{2} L_{\C} ( F_{2} ) p_{2} ) }$ such that $\widetilde{\alpha}^{ 0 , p_{1}L_{\C} ( E_{1} )p_{1} }$ sends $[ p_{1} ]$ to $[ p_{2} ]$.  By Proposition~\ref{p:algkthyfullcorner}, $\iota^{ \mathrm{alg}}$ induces an isomorphism from $\oikalg ( p_{i} L_{\C} ( F_{i} ) p_{i} )$ to $\oikalg ( L_{\C} ( F_{i} ) )$.  Composing isomorphisms gives an isomorphism $\ftn{ \lambda }{ \oikalg ( L_{\C} ( F_{1} ) ) }{ \oikalg ( L_{\C} ( F_{2} ) ) }$ such that $\lambda^{ 0 , L_{\C} ( F_{1} ) } ( [ p_{1} ] ) = [ p_{2} ]$.   

Since $F_{i}$ for $i=1,2$ is a row-finite graph, by Theorem~\ref{t:induceiso}, there exists an isomorphism $\ftn{ \delta }{ \oiktop (  C^{*} ( F_{1} ) ) }{ \oiktop ( C^{*} ( F_{2} ) ) }$ such that 
\begin{align*}
 \delta^{ 0 , C^{*} ( F_{1} ) } ( [ p_{1} ] ) &= \left( \delta^{ 0 , C^{*} ( F_{1} ) } \circ \gamma_{0}^{F_{1}} \right) ( [ p_{1} ] ) \\
 									&= \left( \gamma_{ 0  }^{F_{2}} \circ \lambda^{ 0 , L_{\C} ( F_{1} ) } \right) ( [ p_{1} ] ) \\
									&= \gamma_{0}^{ F_{2} } ( [ p_{2} ] ) \\
									&= [ p_{2} ]
\end{align*}
By Proposition~\ref{p:topkthyfullcorner}, $\iota^{ \mathrm{top}}$ induces an isomorphism from $\oiktop ( p_{i} C^{*} ( F_{i} ) p_{i} )$ to $\oiktop ( C^{*} ( F_{i} ) )$.  Thus, there exists an isomorphism 
\begin{align*}
\ftn{ \zeta }{ \oiktop (  p_{1} C^{*} ( F_{1} ) p_{1} ) }{ \oiktop ( p_{2}C^{*} ( F_{2} ) p_{2} ) }
\end{align*}
such that $\zeta^{ 0 , C^{*} ( F_{1} ) } ( [ p_{1} ] ) = [ p_{2} ]$.  Since $C^{*} ( E_{i} ) \cong p_{i} C^{*} ( F_{i} ) p_{i}$,  there exists an isomorphism $\ftn{ \beta }{ \oiktop ( C^{*} ( E_{1} ) ) }{ \oiktop ( C^{*} ( E_{2} ) ) }$ such that $\beta^{ 0 , C^{*} ( E_{1} ) }$ sends $[ 1_{ C^{*}( E_{1} ) } ]$ to $[ 1_{ C^{*} ( E_{2} ) } ]$.

\end{proof}

\begin{theor} \label{unital-implication-thm}
Let $\mathcal{C}_1$ be a class of graphs that satisfies the following two properties:
\begin{itemize}
\item[(1)] Every graph in $\mathcal{C}_1$ has finitely many vertices and satisfies Condition~(K).
\item[(2)] If $E, F \in \mathcal{C}_1$ and there exists an isomorphism 
\begin{align*}
\ftn{ \alpha }{ \oiktop( C^{*} ( E ) ) }{ \oiktop ( C^{*} ( F) ) }
\end{align*}
such that $\alpha_{0}^{ 0 , C^*(E) }$ sends $[ 1_{ C^{*} ( E ) } ]$ to $[ 1_{ C^{*} ( F ) } ]$, then $C^*(E) \cong C^*(F)$ (as $*$-algebras).
\end{itemize}
Then the isomorphism conjecture holds for all graphs in $\mathcal{C}$.  In other words, if $E, F \in \mathcal{C}$ and $L_{\C} ( E ) \cong L_{\C} ( F )$ (as rings), then $C^{*} ( E ) \cong C^{*} ( F )$ (as $*$-algebras).
\end{theor}

\begin{proof}
Suppose $E, F \in \mathcal{C}_1$ and $L_{\C} ( E ) \cong L_{\C} ( F )$ (as rings).  Then the ring isomorphism from $L_{\C} ( E )$ to $L_{\C} ( F )$ induces an isomorphism 
$$\ftn{ \alpha }{ \oikalg  ( L_{\C} ( E ) ) }{ \oikalg ( L_{\C} ( F ) ) }$$ 
with $\alpha^{ 0 , L_{\C} ( E ) } ([ 1_{ L_{\C} ( E) } ]) = [ 1_{ L_{\C} ( F ) } ]$.   By Property~(1) $E$ and $F$ each have finitely many vertices and satisfy Condition~(K), and thus Theorem~\ref{t:induceisounit} implies there exists an isomorphism 
$$\ftn{ \beta }{ \oiktop ( C^{*} ( E ) ) }{ \oiktop ( C^{*} ( F ) ) }$$ 
with $\beta^{ 0, C^{*} (E) } ([ 1_{ C^{*}( E ) } ])= [ 1_{ C^{*} ( F ) } ]$.  Hence, it follows from Property~(2) that $C^{*} ( E ) \cong C^{*} ( F )$.
\end{proof}

We shall now obtain corollaries to Theorem~\ref{unital-implication-thm} showing that the isomorphism conjecture for graph algebras holds for various classes of graphs where a classification up to isomorphism has been obtained for the associated $C^*$-algebras.  Unlike the corollaries to Theorem~\ref{t:meqstrmeq}, where we had numerous classes of graph $C^*$-algebras where classification up to Morita equivalence (equivalently, stable isomorphism) had been obtained, we have fewer classifications up to isomorphism.  In fact, there are only three classes we are able to discuss: unital $C^*$-algebras of amplified graphs, until graph $C^*$-algebras with exactly one ideal, and until graph $C^*$-algebras that are simple.

\begin{corol} \label{unital-class-results-cor}
Let $E$ and $F$ be graphs, and suppose that $L_{\C} ( E ) \cong L_{\C} ( F )$ (as rings).  If any one of the following bulleted points is true:
\begin{itemize}
\item $E$ and $F$ are each amplified graphs with finitely many vertices.
\item $E$ and $F$ are each graphs with finitely many vertices and $C^*(E)$ and $C^*(F)$ each have exactly one proper, nonzero ideal.
\item $E$ and $F$ are each graphs with finitely many vertices and both $C^*(E)$ and $C^*(F)$ are simple.
\end{itemize}
then $C^{*} ( E ) \cong C^{*} ( F )$ (as $*$-algebras).
\end{corol}

\begin{proof}
We first prove the case when the first bullet point is satisfied.  Let $$\mathcal{C}_1 = \{ E : \text{$E$ is an amplified graph with finitely many vertices} \}.$$  Then every graph in $\mathcal{C}_1$ has finitely many vertices and satisfies Condition~(K), so $\mathcal{C}_1$ satisfies Property~(1) of Theorem~\ref{unital-implication-thm}.  It follows from the results of \cite[Theorem~5.7]{seaser:amplified} that $C^*$-algebras of amplified graphs with finitely many vertices are classified up to isomorphism by ideal-related topological $K$-theory together with the positions of the units, and thus $\mathcal{C}_1$ satisfies Property~(2) of Theorem~\ref{unital-implication-thm}.

For the second and third bullet points we use a similar strategy to apply Theorem~\ref{unital-implication-thm}; namely, we define a collection $\mathcal{C}_1$ of graphs based on properties stated in the bullet point, and then show that that collection satisfies Property (1) and Property (2) in the hypotheses of Theorem~\ref{unital-implication-thm}.  

In the second bullet point, $E$ and $F$ have finitely many vertices by hypothesis and the fact that each of $C^*(E)$ and $C^*(F)$ have finitely many ideals implies that $E$ and $F$ each satisfy Condition~(K), so that Property~(1) holds.  Property~(2) holds due to recent result of Eilers, Restorff, and the first named author in \cite{segrer:scecc} together with the results in \cite{grer:rccconiII}, which show that unital graph $C^*$-algebras with one proper, nonzero ideal are classified up to isomorphism by ideal-related topological $K$-theory together with the positions of the units.

In the third bullet point, $E$ and $F$ have finitely many vertices by hypothesis and the fact that each of $C^*(E)$ and $C^*(F)$ are simple implies that $E$ and $F$ each satisfy Condition~(K), so that Property~(1) holds.  Because any simple graph $C^*$-algebra is either AF or purely infinite, Elliott's theorem and the Kirchberg-Phillips classification theorem imply that unital simple graph $C^*$-algebras are classified up to isomorphism by their $K$-groups together with the position of the unit in the $K_0$-group.  Since ideal-related $K$-theory reduces to the $K$-groups for simple $C^*$-algebras, Property~(2) holds.
\end{proof}

\section{Classification of Leavitt path algebras of amplified graphs} \label{amplified-LPA-classification-sec}

In this section we use our prior results to prove a classification theorem for Leavitt path algebras.  We first need the following theorem which is the analogue of the $C^*$-algebra fact stated in \cite[Theorem~3.8]{seaser:amplified}.  

\begin{theor} \label{moveT}
Let $\alpha = \alpha_1 \alpha _2 \cdots \alpha_n$ be a path in a graph $E$. 
Let $F$ be the graph with vertex set $E^0$, edge set
 \[
	F^1 = E^1 \cup \{ \alpha^m \mid m \in \N \},
\]
and range and source maps that extend those of $E$ and have $r_{F}(\alpha^m) = r_{E}(\alpha)$ and $s_{F}(\alpha^m) = s_{E}(\alpha)$.
If 
\[
	|s_{E}^{-1}(s_{E}( \{ \alpha_1 \}) ) \cap r_{E}^{-1}(r_{E}( \{ \alpha_1 \}))| = \infty,
\]
then there exists $K$-algebra isomorphism from $L_{K}(E)$ to $L_{K}(F)$.
\end{theor}

The proof of Theorem~\ref{moveT} is nearly identical to the proof of \cite[Theorem~3.8]{seaser:amplified}, using  \cite[Theorem~3.7]{AMMS} in place of the gauge-invariant uniqueness theorem, and therefore we omit it.

\begin{defin}
Let $E = ( E^{0}, E^{1} , r_{E} , s_{E} )$ be a graph.
The \emph{amplification of $E$}, denoted by $\overline{E}$, is the graph defined by $\overline{E}^{0} := E^{0}$, 
\begin{align*}
\overline{E}^{1} := \setof{ e(v,w)^{n} }{ \text{$n \in \N$, $v, w \in E^{0}$ and there exists an edge from $v$ to $w$}}, 
\end{align*}
and $s_{ \overline{E} } ( e(v,w)^{n} ) := v$, and $r_{\overline{E}} ( e(v,w)^{n} ) := w$. 
\end{defin}

\begin{defin}
Let $E = ( E^{0}, E^{1} , r_{E} , s_{E} )$ be a graph.
We define the \emph{transitive closure} of $E$ to be the graph ${\tt t}E$ given by:
\begin{align*}
	{\tt t}E^{0} &:= E^{0}, \\
	{\tt t}E^{1} &:= E^1 \cup \setof{ e(v, w ) }{ \text{there is a path but no edge from } v \text{ to } w },
\end{align*}
with range and source maps that extend those of $E$ and satisfy
\begin{align*}
s_{ {\tt t}E} ( e(v,w) ) &:= v,   \\
r_{ {\tt t}E } ( e(v,w) ) &:= w.
\end{align*}
\end{defin}

\begin{corol} \label{transitiveClosure}
If $E$ is a graph with $|E^0|<\infty$, then there exists a $K$-algebra isomorphism $L_{K}(\overline{E})$ to $L_{K}(\overline{{\tt t}E})$.
\end{corol}

\begin{proof}
For any path $\alpha$ in an amplified graph, we have 
\[
	|s_{G}^{-1}(s_{G}( \{ \alpha_1 \}) ) \cap r_{G}^{-1}(r_{G}( \{ \alpha_1 \})) | = \infty,
\]
so Theorem~\ref{moveT} applied a finite number of times proves the desired result.
\end{proof}

\begin{defin}
Let $E = ( E^{0} , E^{1} , r_{E} , s_{E} )$ and $F = ( F^{0} , F^{1} , r_{F} , s_{F} )$.  We say that $E$ and $F$ are \emph{isomorphic}, and write $E \cong F$, if there exist bijections $\ftn{ \alpha^{0} }{ E^{0} }{ F^{0} }$ and $\ftn{ \alpha^{1} }{ E^{1} }{ F^{1} }$ such that 
\begin{align*}
r_{F} ( \alpha^{1} ( e ) ) = \alpha^{0} ( r_{E} ( e ) ) \quad \text{and} \quad s_{F} ( \alpha^{1} ( e ) ) = \alpha^{0} ( s_{E} ( e ) ).
\end{align*} 
\end{defin}

\begin{theor}\label{t:moves}\cite[Theorem~5.7]{seaser:amplified}
Let $E_{1}$ and $E_{2}$ be graphs with finitely many vertices.  Then the following are equivalent.
\begin{itemize}
\item[(a)] $C^{*} ( \overline{E_{1}} ) \cong C^{*} ( \overline{E_{2}} )$ (as $*$-algebras).

\item[(c)] $C^{*} ( \overline{{\tt t}E_{1}} ) \cong C^{*} ( \overline{{\tt t}E_{2}} )$ (as $*$-algebras).

\item[(d)] $\overline{ {\tt t} E_{1} } \cong \overline{ {\tt t} E_{2} }$.

\item[(e)] $\oiktop( C^{*} ( \overline{E_{1}} )  ) \cong \oiktop ( C^{*} ( \overline{E_{2}} ) )$.
\end{itemize} 
\end{theor}

\noindent We shall now prove an analogue of \cite[Theorem~5.7]{seaser:amplified} for Leavitt path algebras.  This result also shows that a converse to the first bullet point of Corollary~\ref{unital-class-results-cor} holds.

\begin{theor} \label{amplified-class-thm}
Let $E_{1}$ and $E_{2}$ be graphs with finitely many vertices.  Then the following are equivalent.
\begin{itemize}
\item[(a)]  $\overline{ {\tt t} E_{1} } \cong \overline{ {\tt t} E_{2} }$. 
 
\item[(b)] $L_{\C} ( \overline{E_{1}} ) \cong L_{\C} ( \overline{E_{2}} )$ (as $\C$-algebras).

\item[(c)] $L_{\C} ( \overline{E_{1}} ) \cong L_{\C} ( \overline{E_{2}} )$ (as rings).

\item[(d)] $\oikalg( L_{\C} ( \overline{E_{1}} )  ) \cong \oikalg ( L_{\C} ( \overline{E_{2}} ) )$.

\item[(e)] $\oiktop( C^{*} ( \overline{E_{1}} )  ) \cong \oiktop ( C^{*} ( \overline{E_{2}} ) )$.

\item[(f)] $C^{*} ( \overline{E_{1}} ) \cong C^{*} ( \overline{ E_{2} } )$ (as $*$-algebras).
\end{itemize}
\end{theor}

\begin{proof}
To see $(a) \Rightarrow (b)$, suppose $\overline{ {\tt t} E_{1} } \cong \overline{ {\tt t} E_{2} }$.  Then $L_{\C} ( \overline{{\tt t} E_{1}} ) \cong L_{\C} ( \overline{ {\tt t }E_{2}} )$ as $\C$-algebras.   By Corollary~\ref{transitiveClosure}, $L_{\C} ( \overline{E_{i}} ) \cong L_{\C} ( \overline{ {\tt t} E_{i} } )$ as $\C$-algebras. Hence, $L_{\C} ( \overline{E_{1}} ) \cong L_{\C} ( \overline{E_{2}} )$ as $\C$-algebras.  Thus $(b)$ holds. 

The implications $(b)  \Rightarrow (c)$ and $(c)  \Rightarrow (d)$ are clear.  To see $(d) \Rightarrow (e)$, suppose $\oikalg( L_{\C} ( \overline{E_{1}} )  ) \cong \oikalg ( L_{\C} ( \overline{E_{2}} ) )$.  By Lemma~\ref{corner-lemma}, Proposition~\ref{p:algkthyfullcorner}, and Proposition~\ref{p:topkthyfullcorner}, 
\begin{align*}
\oikalg( L_{\C} ( \overline{E_{i}} )  ) \cong \oikalg( L_{\C} ( F_{i})  ) \text{ and } \oiktop( C^{*} ( \overline{E_{i}} )  ) \cong \oiktop( C^{*} ( F_{i} )  ),
\end{align*}
where $F_{i}$ is a desingularization of $\overline{E_{i}}$, for each $i = 1,2$.  Therefore, 
\begin{align*}
\oikalg( L_{\C} ( F_{1})  ) \cong \oikalg( L_{\C} ( F_{2})  ).
\end{align*}
By Theorem~\ref{t:induceiso}, $\oiktop( C^{*} ( F_{1} )  ) \cong \oiktop( C^{*} (F_{2} )  )$.  Hence, 
\begin{align*}
\oiktop( C^{*} ( \overline{E_{1}} )  ) \cong \oiktop( C^{*} ( \overline{E_{2}} )  )
\end{align*}
and $(e)$ holds.   

Lastly, $(e) \Rightarrow (f)$ and $(f) \Rightarrow (a)$ follow from Theorem~\ref{t:moves}.
\end{proof}


\begin{thebibliography}{10}


\bibitem{aalp:classleavitt}
{\sc G.~Abrams, P.~N. {\'A}nh, A.~Louly, and E.~Pardo}, {\em The classification
  question for {L}eavitt path algebras}, J. Algebra, 320 (2008),
  pp.~1983--2026.

\bibitem{AbrPino3}
{\sc G.~Abrams and G.~Aranda-Pino}, {\em The Leavitt path algebras of arbitrary graphs}, Houston J.~Math, 34 (2008), pp.~423--442.

\bibitem{gamt:isomorita}
{\sc G.~Abrams and M.~Tomforde}, {\em Isomorphism and {M}orita equivalence of
  graph algebras}, Trans. Amer. Math. Soc., 363 (2011), pp.~3733--3767.

\bibitem{abc:kthyleavitt}
{\sc P.~Ara, M.~Brustenga, and G.~Corti{\~n}as}, {\em {$K$}-theory of {L}eavitt
  path algebras}, M\"unster J. Math., 2 (2009), pp.~5--33.

\bibitem{agop:exhangerings}
{\sc P.~Ara, K.~R. Goodearl, K.~C. O'Meara, and E.~Pardo}, {\em Separative
  cancellation for projective modules over exchange rings}, Israel J. Math.,
  105 (1998), pp.~105--137.

\bibitem{amp:nonstablekthy}
{\sc P.~Ara, M.~A. Moreno, and E.~Pardo}, {\em Nonstable {$K$}-theory for graph
  algebras}, Algebr. Represent. Theory, 10 (2007), pp.~157--178.

\bibitem{AMMS}
{\sc G.~Aranda Pino, D.~Mart\'in Barquero, C.~Mart\'in Gonz\'alez, and M.~Siles Molina}, \emph{Socle theory for
Leavitt path algebras of arbitrary graphs}, Rev. Mat. Iberoam. \textbf{26} (2010), 611--638.

\bibitem{tbdp:flow}
{\sc T.~Bates and D.~Pask}, {\em Flow equivalence of graph algebras}, Ergodic
  Theory Dynam. Systems, 24 (2004), pp.~367--382.

\bibitem{rasmus}
{\sc R.~Bentmann}, {\em Filtrated {$K$}-theory and classification of {$C\sp{\ast}$}-algebras}, Master's thesis, {http://math.ku.dk/$\sim$bentmann/thesis.pdf}, 2010.

\bibitem{rasmusmanuel}
{\sc R.~Bentmann and M.~K\"ohler}, {\em Universal coefficient theorems for {$C\sp{\ast}$}-algebras over
  finite topological spaces}, arXiv:1101.5702v3 [math.OA].
  
\bibitem{Boyle-Huang}  
{\sc  M.~Boyle and D.~Huang}, {\em Poset block equivalence of integral matrices}, Trans. Amer. Math. Soc., 355 (2003), pp.~3861--3886.

\bibitem{heralgs}
{\sc L.~G. Brown}, {\em Stable isomorphism of hereditary subalgebras of
  {$C\sp*$}-algebras}, Pacific J. Math., 71 (1977), pp.~335--348.

\bibitem{gc:algkthy}
{\sc G.~Corti{\~n}as}, {\em Algebraic v. topological {$K$}-theory: a friendly
  match}, in Topics in algebraic and topological {$K$}-theory, vol.~2008 of
  Lecture Notes in Math., Springer, Berlin, 2011, pp.~103--165.

\bibitem{dhs:strgrphalg}
{\sc K.~Deicke, J.~H. Hong, and W.~Szyma{\'n}ski}, {\em Stable rank of graph
  algebras. {T}ype {I} graph algebras and their limits}, Indiana Univ. Math.
  J., 52 (2003), pp.~963--979.

\bibitem{DT1}
{\sc D.~Drinen and M.~Tomforde}, \emph{The $C^*$-algebras of
arbitrary graphs}, Rocky Mountain J.~Math. \textbf{35} (2005),
pp.~105--135.

\bibitem{segrer:ccfis}
{\sc S.~Eilers, G.~Restorff, and E.~Ruiz}, {\em Classifying {$C^*$}-algebras
  with both finite and infinite subquotients}.
\newblock Submitted, arXiv: 1009.4778.

\bibitem{segrer:okfe}
\leavevmode\vrule height 2pt depth -1.6pt width 23pt, {\em The ordered
  {$K$}-theory of a full extension}.
\newblock Submitted, ArXiV: 1106.1551.

\bibitem{segrer:scecc}
\leavevmode\vrule height 2pt depth -1.6pt width 23pt, {\em Strong
  classification of extensions of classifiable {$C^*$}-algebras}.
\newblock In preparation.

\bibitem{ERRshift}
\leavevmode\vrule height 2pt depth -1.6pt width 23pt, {\em Classification of
  extensions of classifiable {$C^\ast$}-algebras}, Adv. Math., 222 (2009),
  pp.~2153--2172.

\bibitem{ERRlinear}
\leavevmode\vrule height 2pt depth -1.6pt width 23pt, {\em On graph
  {$C^*$}-algebras with a linear ideal lattice}, Bull. Malays. Math. Sci. Soc.
  (2), 33 (2010), pp.~233--241.


\bibitem{seaser:amplified}
{\sc S.~Eilers, A.~S{\o}rensen, and E.~Ruiz}, {\em Amplified graph {$C\sp \ast$}-algebras}.
\newblock To appear in M\"{u}nster J. Math.  Preprint ArXiv:1110.2758.

\bibitem{semt_classgraphalg}
{\sc S.~Eilers and M.~Tomforde}, {\em On the classification of nonsimple graph
  {$C^*$}-algebras}, Math. Ann., 346 (2010), pp.~393--418.

\bibitem{filtrated}
{\sc R.~Meyer and R.~Nest}, {\em {$C\sp{\ast}$}-algebras over topological spaces: Filtrated
  {$K$}-theory},  Canad. J. Math. 64 (2012), no. 2, pp.~368--408.
  
 
\bibitem{gr:ckalg}
{\sc G.~Restorff}, {\em Classification of {C}untz-{K}rieger algebras up to
  stable isomorphism}, J. Reine Angew. Math., 598 (2006), pp.~185--210.

\bibitem{grer:rccconiII}
{\sc G.~Restorff and E.~Ruiz}, {\em On {R\o rdam's} classification of certain
 {$C^*$}-algebras with one nontrivial ideal II}, Math. Scand., 101 (2007),
pp.~280--292.

\bibitem{ermt:ideals-graph-algs}
{\sc E.~Ruiz and M.~Tomforde}, {\em Ideals in graph algebras}, Submitted for publication, Preprint version ArXiv:1205.1247v1.

\bibitem{as:geo}
{\sc A.~S{\o}rensen}, {\em Geometric classification of simple graph algebras}.
\newblock arXiv:1111.1592.

\bibitem{as:excisionkthy}
{\sc A.~A. Suslin}, {\em Excision in integer algebraic {$K$}-theory}, Trudy
  Mat. Inst. Steklov., 208 (1995), pp.~290--317.
\newblock Dedicated to Academician Igor{\cprime} Rostislavovich Shafarevich on
  the occasion of his seventieth birthday (Russian).

\bibitem{sw:excisionalgktheory}
{\sc A~A. Suslin and M. Wodzicki}, {\em Excision in algebraic {$K$}-theory}, Ann. of Math. (2) 136 (1992), pp. 51--122. 

\bibitem{mt:idealst}
{\sc M.~Tomforde}, {\em Uniqueness theorems and ideal structure for {L}eavitt
  path algebras}, J. Algebra, 318 (2007), pp.~270--299.

\bibitem{cb:homtopalgkthy}
{\sc C.~A. Weibel}, {\em Homotopy algebraic {$K$}-theory}, in Algebraic
  {$K$}-theory and algebraic number theory ({H}onolulu, {HI}, 1987), vol.~83 of
  Contemp. Math., Amer. Math. Soc., Providence, RI, 1989, pp.~461--488.
  
\bibitem{mw:excisioncyclichom}
{\sc M.~Wodzicki}, {\em Excision in cyclic homology and in rational algebraic {$K$}-theory}, Ann. of Math. (2) 129 (1989), pp. 591--639.  

\end{thebibliography}
\end{document}